\documentclass[a4paper]{amsart}
\usepackage[latin1]{inputenc}
\usepackage{amsfonts}
\usepackage{amsmath,latexsym,amssymb,amsfonts, amsthm}
\usepackage{esint}
\usepackage{cite}
\usepackage{graphicx}
\usepackage{amscd}
\usepackage{color}
\usepackage{bm}             
\usepackage{enumerate}
\usepackage[dvips]{epsfig}
\usepackage{psfrag}
\usepackage{epsfig}
\usepackage{comment}

\usepackage[
  hmarginratio={1:1},     
  vmarginratio={1:1},     
  textwidth=16cm,        
  textheight=21cm,
  heightrounded,          
]{geometry}

\usepackage{tikz}
\usepackage{graphicx,color}
\usepackage[colorlinks]{hyperref}
\hypersetup{linkcolor=blue,citecolor=blue,filecolor=black,urlcolor=blue}

\newcommand{\lp}[2]{\|#1\|_{L^{#2}(\cG)}}

\newcommand{\beq}{\begin{equation}}
\newcommand{\eeq}{\end{equation}}

\newtheorem{thm}{Theorem}
\newtheorem{theorem}[thm]{Theorem}
\newtheorem{proposition}[thm]{Proposition}

\newtheorem{cor}[thm]{Corollary}
\newtheorem{lem}[thm]{Lemma}
\newtheorem{lemma}[thm]{Lemma}

\theoremstyle{definition}
\newtheorem{rem}[thm]{Remark}

\newtheorem{definition}[thm]{Definition}

\numberwithin{thm}{section}

\numberwithin{equation}{section}

\newcommand{\R}{\mathbb{R}}
\newcommand{\cG}{\mathcal{G}}
\newcommand{\Gcal}{\mathcal{G}}

\newcommand{\cE}{\mathcal{E}}
\renewcommand{\a}{\alpha}

\title[Ground states for the NLSE with combined nonlinearity on periodic metric graphs]{Ground states for the NLS equation with combined nonlinearity on periodic metric graphs}
\author{Nicola Soave and Lorenzo Villata}
\address{Nicola Soave and Lorenzo Villata\newline\indent
Dipartimento di Matematica ``Giuseppe Peano''
\newline\indent
Universit\`a degli Studi di Torino
\newline\indent
Via Carlo Alberto 10, 10123, Torino, Italy   }
\email{nicola.soave@unito.it; lorenzo.villata@unito.it}

\thanks{Nicola Soave is a member the INDAM-GNAMPA group, and is partially supported by the PRIN Project no.~2022R537CS ``$NO^3$ -- Nodal optimization, NOnlinear elliptic equations, NOnlocal geometric problems, with a focus on regularity'', CUP D53D23005930006, funded by the European Union - Next Generation EU within the PRIN 2022 program (D.D. 104 - 02 02 2022 Ministero dell'Universit\`a e della Ricerca, Italy).}

\begin{document}

\begin{abstract}
We investigate the existence of ground states with prescribed mass for the Non-Linear Schr\"odinger energy with combined nonlinearities on $1$ and $2$-periodic metric graphs. This is the natural prosecution of previous studies concerning on the one hand the homogeneous NLS equation on periodic graphs, and on the other hand the NLS equation with combined nonlinearity on noncompact metric graphs with finitely many vertexes and edges. As in the latter case, it turns out that the interplay between different nonlinearities creates new phenomena with respect to the homogenous setting, but, due to the periodicity, in a quite different way; in particular, for $2$-periodic graphs, the so called dimensional crossover occurs. 

As a by-product, we extend existing results for the homogeneous NLS on the square and honeycomb grids to general $2$-periodic graphs. Furthermore, we also improve previous results obtained for the inhomogeneous NLS on noncompact graphs with finitely many vertexes and edges.
\end{abstract}

\maketitle

\section{Introduction}
In this paper we investigate the existence of ground states for the NLS energy with combined nonlinearities
\beq\label{def E}
E_{p,q,\alpha}(u, \cG) := \int_{\cG} \left(\frac12 |u'|^2 - \frac1p|u|^p - \frac{\alpha}{q}|u|^q\right)dx,
\eeq
under the mass constraint 
\[
u \in H^1_\mu(\cG) := \left\{ u \in H^1(\cG): \ \int_{\cG} |u|^2\,dx = \mu \right\},
\]
where $\a \in \R$, $2<q<p \le 6)$, and $\cG$ is a non-compact \emph{periodic} metric graph. A metric graph is a connected metric space obtained by glueing together a finite number of closed line intervals, the \emph{edges} of the graph, by identifying some of their endpoints. The endpoints are the \emph{vertices} of the graph. Any bounded edge $\mathrm{e}$ is identified with a closed bounded interval $[0,\ell_{\mathrm{e}}]$ (where $\ell_{\mathrm{e}}$ is the length of $\mathrm{e}$), while unbounded edges are identified with (a copy of) the closed half-line $[0,+\infty)$. In this paper we focus on $1$ and $2$-\emph{periodic} metric graphs; we refer to Definition \ref{def: periodic graph} below, but the idea is that a graph $\cG$ is $1$ or $2$-periodic if it is built of an infinite number of copies of a fixed compact graph, the periodicity cell, glued together along one or two non-parallel directions. Prototypical examples are the ladder-type graph, and the square grid.

The current investigation is the natural prosecution of \cite{combined}, where the NLS energy with combined nonlinearity on noncompact metric graphs with a finite number of edges and vertexes was considered. Any such graph contains at least a half-fine, while here by periodicity we assume that $\cG$ has infinitely many edges and vertexes, and no half-line. This marks a significant difference with respect to \cite{combined}, as already observed in the homogeneous case, which was studied in \cite{periodic, square_grid, honeycomb, spatial_grid} in the periodic case, and in \cite{2015, threshold, critical} in the noncompact non-periodic case (we also refer to \cite{Do, DoTe, PiSoVe, SeTe} for related results).

With $\Gcal$ as above, a function $u: \Gcal \to \R$ can be identified with a vector of functions $\{u_{\mathrm{e}}\}$, where each $u_{\mathrm{e}}$ is defined on the corresponding interval $[0,\ell_{\mathrm{e}}]$ (or $[0,+\infty)$ if $\mathrm{e}$ is unbounded). Endowing each edge with Lebesgue measure, one can define $L^p$ spaces over $\cG$, denoted by $L^p(\cG)$, in a natural way, with norm
\[
\|u\|_{L^p(\cG)}^p = \sum_{\mathrm{e}} \|u_{\mathrm{e}}\|_{L^p(\mathrm{e})}^p.
\]
The Sobolev space $H^1(\cG)$ is defined as the set of functions $u: \cG \to \R$ such that $u_{\mathrm{e}} \in H^1([0, \ell_{\mathrm{e}}])$ for every $\mathrm{e}$, and $u$ is continuous on $\cG$ (in particular, if a vertex $\mathrm{v}$ belongs to more than one edge, the corresponding functions $u_{\mathrm{e}}$ take the same value on $\mathrm{v}$); the norm in $H^1(\cG)$ is naturally defined as
\[
\|u\|_{H^1(\cG)}^2 = \sum_{\mathrm{e}} \|u_{\mathrm{e}}'\|_{L^2(\mathrm{e})}^2 + \|u_{\mathrm{e}}\|_{L^2(\mathrm{e})}^2.
\]
In this framework, by \emph{a ground state of mass $\mu$} we mean a minimizer for the problem
\beq\label{def gs level}
\mathcal{E}_{p,q,\alpha}(\mu,\cG):= \inf_{u \in H^1_\mu(\cG)} E_{p,q,\alpha}(u, \cG).
\eeq
The value $\mathcal{E}_{p,q,\alpha}(\mu,\cG)$ is called \emph{ground state energy level}, and it is clear that, searching for ground states, it is sufficient to work with real valued functions. \\
Ground states, and other critical points of $E_{p,q,\alpha}(\cdot,\Gcal)$ constrained on $H^1_{\mu}(\Gcal)$, satisfy, for some $\lambda \in \R$, the stationary NLS equation
\begin{equation}\label{eq:NLSE}
-u'' +  \lambda u = |u|^{p-2} u + \alpha |u|^{q-2} u 
\end{equation}
on every edge; moreover, at each vertex the Kirchhoff condition is satisfied, which requires
 the sum of all the outgoing derivatives to vanish (see \cite[Proposition 3.3]{2015}). The study of nonlinear Schr\"odinger equations on metric graphs has attracted considerable attention in the last decade. From the mathematical point of view, the problem presents a number of interesting new features with respect to the classical Euclidean setting (see e.g. \cite{2015, threshold, critical}). Furthermore, nonlinear evolution on graphs turns out to be relevant also from the physical point of view (see e.g. \cite{Meh, BoCa, SMSS}). Regarding Schr\"odinger equations on periodic graphs, we refer the interested reader to \cite{Ex, ExTu} and to \cite[Chapter 4]{intro_graphs} for the linear equation. Nonlinear problems have been studied in \cite{PeSc}, where the authors considered a specific example of $1$-periodic graph; in \cite{Pank}, where a \emph{fixed-frequency approach} via the Nehari manifold is used on general periodic graphs (such an approach is not directly effective to deal with the fixed mass problem); in \cite{periodic}, where the fixed-mass homogeneous problem is studied on general $1$-periodic graphs; in \cite{square_grid} and \cite{honeycomb}, where the fixed-mass homogeneous problem is studied on the (two-dimensional) square and hexagonal grids, respectively; and to \cite{spatial_grid}, which concerns the problem on the spatial cubic grid. 
 
Concerning the NLS with combined nonlinearities, the literature of the topic is huge, and we only provide a brief and incomplete list of contributions (we refer the interested reader to references therein for a more comprehensive picture): we refer to \cite{TaViZh} for a detailed analysis of its dynamical properties in $\R^d$; to \cite{JJLV, JeLe, So1, So2, WeWu} for existence and stability of bound and ground states in $\R^d$ ; to \cite{combined} for existence of ground states on noncompact metric graphs with finitely many edges and vertexes; to \cite{LiZhLi} for the same problem with localized nonlinearity; and to \cite{BBDT, BoDo0, BoDo, LiX24} to doubly nonlinear equations with a standard nonlinearity combined with a pointwise $\delta$ nonlinearity. In all these contributions, it emerges that the interplay between different nonlinear terms creates a richer picture with respect to the homogeneous case.

 \medskip

Let us now discuss in more detail the results that are directly related to our study. The starting points of our investigation are on one side the results for the associated homogeneous problem on periodic graphs; and, on the other side, the study of the NLS energy with combined nonlinearity on noncompact graphs with a finite number of edges and vertexes. 

The homogeneous problem is obtained by taking $\alpha=0$ in \eqref{def E}. Note that the choice of $q$ is irrelevant, and we use the notation
\beq\label{homog func}
E_{p}(u,\cG):= E_{p,q,0}(u,\cG) \quad \text{and} \quad \cE_p(\mu,\cG) = \cE_{p,q,0}(\mu, \cG).
\eeq
 %
%
%
%
It is well known that the so called $L^2$-critical exponent $p=6$ plays an important role, distinguishing the subcritical case $p<6$, where the functional $E_p$ is bounded from below in $H^1_{\mu}(\mathcal{G})$ for every $\mu>0$, from the supercritical case $p>6$, where the infimum of $E_p$ in $H^1_{\mu}(\mathcal{G})$ is $-\infty$ for every $\mu$; this phenomenon is strictly related to the Gagliardo-Nirenberg inequality, and is typical of one-dimensional problems with $L^2$-constraints. In higher dimensions $d \ge 2$, the $L^2$-critical exponent is $2+4/d$. For $p=6$, we have instead that $E_p$ is bounded from below on $H^1_\mu(\cG)$ only for masses $\mu$ smaller than a critical threshold $ \tilde \mu_{\cG}$, depending on $\cG$ through suitable topological properties: for instance, in the case of noncompact graphs with finitely many edges, it is shown in \cite{critical} that
\beq
\label{critical mass combined}
    \tilde{\mu}_{\mathcal{G}}:=\begin{cases}
    \mu_{\mathbb{R}^+}&=\frac{\sqrt{3}}{4}\pi \quad \text{ if $\mathcal{G}$ contains a terminal point,} \\
        \mu_{\mathbb{R}}&=\frac{\sqrt{3}}{2}\pi \quad \text{ if $\mathcal{G}$ does not contain any terminal point;}
    \end{cases}
\end{equation}
we recall that a terminal point is a vertex of degree $1$ (namely a vertex connected with a single edge), and a terminal edge is an edge with a terminal point. The values $\mu_{\R^+}$ and $\mu_{\R}$ are the critical thresholds for the half-line and for the real line, respectively. Furthermore, the infimum $\cE_6(\mu,\cG)$ is not necessarily achieved, even when $\cE_6(\mu,\cG)>-\infty$ (see \cite{critical} again).

Coming back to the periodic setting, the $1$-periodic homogeneous case was investigated in \cite{periodic}, and the $2$-periodic case in \cite{square_grid} and \cite{honeycomb} for the square and the honeycomb grids, respectively. In \cite{periodic}, it was shown that on $1$-periodic graphs the scenario is somehow similar to the one for graphs with finitely many edges: ground state of mass $\mu$ always exists in the subcritical regime $p<6$, while in the critical case the functional $E_6$ can have no minimizers even when it is bounded from below, due to lack of compactness; remarkably, for certain classes of graphs (such as graphs with a terminal point), the infimum $\cE_6(\mu, \cG)$ is never attained, regardless of the value of $\mu$ (see Theorem \ref{1_periodic_critical} below).

The situation is substantially more subtle for the problem on the square or the honeycomb grids \cite{honeycomb, square_grid}. The local $1$-dimensional nature of the single edges interacts with the global $2$-dimensional nature of the graph, due to the $2$-periodicity: the idea is that, since the metric graph is composed of a single compact core copied along two non-parallel directions, the Gagliardo-Nirenberg inequality can be proved both in the one-dimensional and in the two-dimensional versions, leading to a hybrid behavior. This phenomenon is called ``dimensional crossover", and has been rigorously established on the particular graphs considered in \cite{honeycomb, square_grid}. As a result, in such cases the subcritical regime, in which the ground state exists for every value of $\mu$, occurs only for $p \in (2,4)$ (note that $4$ is the two-dimensional critical exponent); for $p \in [4,6)$, although $\mathcal{E}_p(\cdot,\mathcal{G}) > -\infty$ for every $\mu > 0$ (as a consequence of the one-dimensional Gagliardo-Nirenberg inequality), the infimum is attained only for sufficiently large masses (as a consequence of the two-dimensional Gagliardo-Nirenberg inequality). As for the $L^2$-critical case $p=6$, phenomena similar to those already observed in the $1$-periodic setting arise.

Regarding the existence of ground states for the NLS energy with combined nonlinearity, namely for $\mathcal{E}_{6,q,\a}(\mu,\mathcal{G})$ with $\alpha \in \mathbb{R}$, $q \in (2,6)$, and $\mathcal{G}$ noncompact with a finite number of edges and vertices, the problem was investigated in \cite{combined}. The idea was to understand how the addition of a subcritical perturbation term $\alpha \|u\|_{L^q(\cG)}^q/q$ affects the existence of ground states for the associated critical functional $E_6(\cdot\,,\cG)$, in the spirit of the Brezis-Nirenberg problem; in this perspective, we note that, if $\mathcal{G}$ is noncompact with a finite number of edges and vertices, then the problem associated with $\mathcal{E}_{p,q,\alpha}(\mu,\mathcal{G})$ would be purely subcritical for $2<q<p<6$, and therefore very similar to the homogeneous subcritical case; for this reason, in \cite{combined} the analysis was restricted to the case $p=6$, in contrast with the present study, where the dimensional crossover makes other values of $p$ of interest. It follows from \cite{combined} that two markedly different scenarios arise, depending on whether the perturbation is focusing ($\alpha>0$) or defocusing ($\alpha<0$). In the focusing case, critical and subcritical effects are combined in the following way: as in the critical case, there exists the critical threshold $\tilde \mu_{\cG}$ for the mass, above which the ground state energy level is  $- \infty$. For masses below $\tilde \mu_{\cG}$, ground states may or may not exist, and subcritical methods can often be adapted to answer this question regardless of the precise value of $\alpha>0$ (we refer to Theorem 1.1, Corollaries 1.2 and 1.3, Propositions 1.4 and 1.5 in \cite{combined} for more details). Instead, the defocusing case $\alpha<0$ is somehow less favorable for existence of ground states and is more reminiscent of the purely critical case. There exist entire classes of graphs for which the ground state level is never attained and, even when the homogeneous problem admits a ground state, the inhomogeneous one retains this property only for sufficiently small absolute values of the perturbation coefficient $|\alpha|$ (see Theorems 1.6 and 1.7 in \cite{combined}).

\medskip

The main goal of this work is to analyze the case of combined nonlinearities on $1$ and $2$-periodic graphs, with particular emphasis on understanding the role played by the dimensional crossover. In doing so, we also extend the results obtained for the homogeneous $2$-periodic problem (which are currently available only for the square and hexagonal grids) to a very general class of $2$-periodic graphs. We point out that we shall always focus on the case $2<q<p \le 6$. The case when $p>6$ is not interesting from the point of view of existence of ground states, since it is easy to check that the ground state energy level is always $-\infty$.
%

\medskip

\noindent \textbf{Main results in the focusing case $\alpha>0$.} 
In what follows, we always assume that $2<q<p \le 6$, without explicitly stating it. The first statement describe the lower boundedness of the functional $E_{p,q,\alpha}(\cdot\,,\cG)$ on $H^1_\mu(\cG)$. 

\begin{proposition}
\label{(un)bounded}
Let $\mathcal{G}$ be a $1$ or $2$-periodic graph, and $\alpha>0$. If either $p<6$ and $\mu>0$, or $p=6$ and $\mu \in (0, \tilde{\mu}_{\mathcal{G}})$, then $E_{p,q,\alpha}(\cdot\,,\cG)$ is bounded from below and coercive in $H^1_{\mu}(\mathcal{G})$. Otherwise, if $p=6$ and $\mu\geq \tilde{\mu}_{\mathcal{G}}$, we have that $\mathcal{E}_{p,q,\alpha}(\mu, \mathcal{G})=-\infty$.
\end{proposition}

Now we focus on existence of ground states in all the cases when $\cE_{p,q,\alpha}(\mu,\cG)>-\infty$. We start from the $1$-periodic setting. Similarly to what was already observed in \cite{combined}, it turns out that a focusing subcritical perturbation is favorable for existence. In particular, differently from the homogeneous energy $E_6$, the functional $E_{6,q,\alpha}$ never loses compactness in $H^1_{\mu}(\cG)$ when $\mu<\tilde{\mu}_{\mathcal{G}}$, and this allows us to prove the following:

\begin{thm}
\label{sottocritico_combined}
Let $\mathcal{G}$ be a $1$-periodic metric graph, $\alpha>0$, and either $p<6$ and $\mu>0$, or $p=6$ and $\mu \in (0,\tilde{\mu}_{\mathcal{G}})$. Then 
    \begin{equation*}
        -\infty<\mathcal{E}_{p,q,\alpha}(\mu, \mathcal{G})<0
    \end{equation*} 
    and there exists a ground state of mass $\mu$.
\end{thm}

\begin{rem}
We observe that, unlike the homogeneous case considered in \cite{periodic} and recalled in Theorems \ref{1_periodic_subcritical}, \ref{1_periodic_critical}, the focusing case with combined nonlinearities exhibits a subcritical behavior whenever $E_{p,q,\alpha}$ is bounded from below on $H^1_{\mu}(\cG)$. 
\end{rem}

When dealing with a $2$-periodic metric graph $\mathcal{G}$, the dimensional crossover also takes place for the functional $E_{p,q,\alpha}$. As a result, the purely subcritical case appears only when $q<4$. 
\begin{thm}
\label{sottocritico bidimensionale combined}
Let $\mathcal{G}$ be a $2$-periodic metric graph, $\alpha>0$, $q \in (2,4)$, $p>q$ and either $p<6$ and $\mu>0$, or $p=6$ and $\mu\in (0,\tilde{\mu}_{\mathcal{G}})$. Then 
\begin{equation*}
    -\infty < \mathcal{E}_{p,q,\alpha}(\mu, \mathcal{G})<0
\end{equation*}
and there exists a ground state of mass $\mu$.
\end{thm}


When $q \in [4,6)$, a further critical mass $\mu_{p,q,\alpha,\cG}$ is introduced through Lemma \ref{massa critica esistenza}, and the following theorem holds.
\begin{thm}
\label{teorema 2 sopracritico combined}
Let $\mathcal{G}$ be a $2$-periodic metric graph, $\alpha>0$ and either $4 \leq q<p< 6$ and $\mu>0$, or $4 \leq q<p=6$ and $\mu \in (0,\tilde{\mu}_{\mathcal{G}})$. Then there exists a threshold $\mu_{p,q,\alpha,\cG} >0$, with $\mu_{6,q,\alpha,\cG} < \tilde{\mu}_{\mathcal{G}}$, such that the following alternatives occur:
\begin{enumerate}
\item if $\mu < \mu_{p,q,\alpha,\cG}$, then \begin{equation*}
        \mathcal{E}_{p,q,\alpha}(\mu, \mathcal{G})=0
\end{equation*} 
    and the ground state energy level of mass $\mu$ is not attained;
\item if $\mu > \mu_{p,q,\alpha, \cG}$, then 
    \begin{equation*}
        -\infty < \mathcal{E}_{p,q,\alpha}(\mu, \mathcal{G})<0
    \end{equation*}
    and there exists a ground state of mass $\mu$;
\item if $\mu = \mu_{p,q,\alpha,\cG}$, then 
    \begin{equation*}
     \mathcal{E}_{p,q,\alpha}(\mu_{p,q,\alpha,\cG}, \mathcal{G})=0
    \end{equation*}
    and, for $q>4$, there exists a ground state of mass $\mu_{p,q,\alpha,\cG}$.
\end{enumerate}
\end{thm}
The analysis of the case $q=4$, $\mu=\mu_{p,4,\alpha,\cG}$ represents an open problem. 

We emphasize that the interval of existence for the mass depends on $\alpha$. This represents a novel feature with respect to the results obtained in \cite{combined}, where, for $\alpha>0$, one always had existence or non-existence on mass intervals independent of the size of $\alpha$ itself. In this perspective, it is natural to study the behavior of $\mu_{p,q,\alpha, \cG}$ when $\alpha$ varies. 

\begin{theorem}\label{int alpha}
Under the assumptions of Theorem \ref{teorema 2 sopracritico combined}, let $\mu_{p,q,\alpha,\cG} >0$ as above, and let $\mu_{r,\cG}$ be the critical mass for the exponent $r \in [4,6]$, defined in \eqref{critical mass} and \eqref{critical mass p} below. If $q \in (4,6)$, then we have that
\[
\mu_{p,q,\alpha,\cG} < \min \left\{\mu_{p,\cG}, \ \left(\frac{1}{\alpha}\right)^{\frac{2}{q-2}} \mu_{q,\cG}\right\},
\]
while for $q=4$
\[
\mu_{p,4,\alpha,\cG} < \mu_{p,\cG} \quad \text{and} \quad \mu_{p,4,\alpha,\cG} \le \frac{\mu_{4,\cG}}{\alpha}.
\]
Moreover, we have that $\mu_{p,q,\alpha,\cG} \ge \overline{\mu}_{p,q,\alpha,\cG}$, where the latter is defined as the unique solution to
\[
1-\left(\frac{t }{\mu_{p,\cG}}\right)^{\frac{p-2}{2}}-\alpha \left(\frac{t}{\mu_{q,\cG}}\right)^{\frac{q-2}{2}} = 0.
\]
In particular, $\mu_{p,q,\alpha,\cG} \to 0^+$ as $\alpha \to +\infty$, and $\mu_{p,q,\alpha,\cG} \to \mu_{p,\cG}^-$ as $\alpha \to 0^+$.
\end{theorem}

From a heuristic point of view, this suggests that increasingly larger perturbation parameters $\alpha$
correspond to correspondingly larger intervals of existence. On the other hand, if $\alpha \to 0^+$,
the situation tends to converge to that of the homogeneous problem.

\begin{rem}
In proving Theorems \ref{teorema 2 sopracritico combined} and \ref{sottocritico bidimensionale combined}, we will also use the known results for the homogeneous case $\alpha=0$ on $\cG$. At the moment, such results were only established when $\cG$ is the square \cite{square_grid} or honeycomb \cite{honeycomb} grids. Thus, as further results, we will extend the main theorems in \cite{square_grid, honeycomb} to general $2$-periodic graphs (see Definition \ref{def: periodic graph} and Theorems \ref{sottocritico_2_periodico} and \ref{teorema 2-sopracritico grafi 2-periodici} below). This study is of independent interest.
\end{rem}

\medskip

\noindent \textbf{Main results in the defocusing case $\alpha<0$.}  The following statement, analogous to \cite[Theorems 1.6 and 1.7(i)]{combined}, clarify the lower boundedness and the of loss of campactness of the functional $E_{p,q,\alpha}$. Again, we always suppose that $2<q<p \le 6$.

\begin{proposition}
\label{(un)bounded_defocusing}
Let $\mathcal{G}$ be a $1$ or $2$-periodic metric graph and $\alpha <0$. The following holds:
\begin{enumerate}
\item If $p<6$ and $\mu>0$, then 
\[
    -\infty<\mathcal{E}_{p,q,\alpha}(\mu,\cG)\leq 0,
\]
and $E_{p,q,\alpha}(\cdot\,,\cG)$ is coercive in $H^1_{\mu}(\cG)$. 
\item If $p=6$ and $\mu > \tilde{\mu}_{\mathcal{G}}$, then 
\begin{equation*}
    \mathcal{E}_{p,q,\alpha}(\mu,\cG)=-\infty.
\end{equation*}
\item If $p=6$ and $\mu \in (0, \tilde{\mu}_{\mathcal{G}}]$, then
\[
    -\infty<\mathcal{E}_{6,q,\alpha}(\mu,\cG)\leq 0.
\]
Moreover, if $\mu \in (0, \mu_{6,\mathcal{G}}]$, where $\mu_{6,\cG}$ is defined in \eqref{critical mass} below, then $\mathcal{E}_{6,q,\alpha}(\mu,\cG)=0$ and the ground state level is not achieved.
\end{enumerate}
\end{proposition}
This implies that the existence of ground states is possible either for $p \in (2,6)$, or for $p=6$ with $\mu \in (\mu_{6,\mathcal{G}},\tilde \mu_{\cG}]$ (in analogy with \cite{combined}); we point out that this interval can empty, and in such case we can immediately conclude that ground states never exist for $p=6$. This is the case when $\cG$ has a terminal point, as $\tilde \mu_{\cG}=\mu_{\R^+}=\mu_{6,\cG}$ (for the latter equality, we refer to \cite{critical}). Moreover, $(\mu_{6,\mathcal{G}},\tilde \mu_{\cG}]$ can be empty also when $\cG$ has no terminal point, but some topological condition is verified. For instance, if assumption ($H$) below is verified (namely if the graph \emph{has a cycle covering}) then $\tilde \mu_{\cG}=\mu_{\R}=\mu_{6,\cG}$, see \cite{critical} again. 

%
%

In \cite[Theorem 1.7 (ii)]{combined}, in studying the interesting case $p=6$ with $\mu \in (\mu_{6,\mathcal{G}},\mu_{\mathbb{R}}]$, it is proved the existence of a critical value $\overline{\alpha}=\overline{\alpha}(\mathcal{G}, q, \mu)<0$ such that: 
\begin{itemize}
    \item[($i$)] if $\alpha <\overline{\alpha}$, then $E_{6,q,\alpha}(\cdot\,,\cG)$ loses compactness in $H^1_{\mu}(\cG)$, and the ground state level is not achieved;
    \item[($ii$)] if $\overline{\alpha}<\alpha<0$, then there exists a ground state of negative energy.
\end{itemize}
A similar scenario arise in the present setting. We define $\overline{\alpha}= \overline{\alpha}(\cG, p,q,\mu)$ as
 \begin{equation}
    \label{alpha_definition}
        \overline{\alpha}:=\inf\{\alpha<0 \text{ such that }\mathcal{E}_{p,q,\alpha}(\mu,\cG)<0\}=\sup\{\alpha < 0 \text{ such that } \mathcal{E}_{p,q,\alpha}(\mu,\cG)=0\},
    \end{equation}
being $\overline{\alpha}:=-\infty$ if $\mathcal{E}_{p,q,\alpha}(\mu,\cG)<0$ for every $\alpha<0$, and $\overline{\alpha}:=0$ if $\mathcal{E}_{p,q,\alpha}(\mu,\cG)=0$ for every $\alpha<0$.

\begin{thm}
\label{defocusing_theorem}
    Let $\mathcal{G}$ be a $1$ or $2$-periodic metric graph, and either $p<6$ with $\mu>0$, or $p=6$ with $\mu \in (\mu_{6,\cG}, \mu_{\R}]$. Let also $\bar \alpha$ be defined by \eqref{alpha_definition}. Then:
    \begin{enumerate}
\item $\overline{\alpha}=0$ if and only if $\mathcal{E}_p(\mu,\cG)=0$;
\item if $\overline{\alpha}<0$, then: 
\begin{itemize}
    \item[($i$)] for $\alpha <\overline{\alpha}$, then $E_{p,q,\alpha}(\cdot\,,\cG)$ has no ground state on $H^1_\mu(\cG)$;
    \item[($ii$)] for $\overline{\alpha}<\alpha<0$, then there exists a ground state for $E_{p,q,\alpha}(\cdot\,,\cG)$ on $H^1_\mu(\cG)$.
\end{itemize}
\item It results that $\overline{\alpha}>-\infty$, and moreover there exists a ground state of $E_{p,q,\overline{\alpha}}(\cdot\,,\cG)$ in $H^1_{\mu}$.    \end{enumerate}
\end{thm}

We emphasize that, in \cite{combined}, the existence of ground states for $E_{p,q,\overline{\alpha}}$ in $H^1_{\mu}$ and the fact that $\overline{\alpha}>-\infty$ when $\mu=\mu_{\R}$ were left as open problems, while here we can treat also these cases. In fact, our method can be extended also to non-compact graphs with finitely many edges, giving a definite answer to such open questions in \cite{combined}:

\begin{thm}
\label{alpha critical 1}
    If $\mathcal{G}$ is a noncompact metric graph with a finite number of edges, $p=6$, and $ \mu \in (\mu_{6,\mathcal{G}},  \mu_{\mathbb{R}}]$, then $\overline{\alpha}>-\infty$, and there exists a ground state for $E_{6,q,\overline{\alpha}}$.
\end{thm}
%
%
%

Coming back to the periodic setting, some remarks are in order:

\begin{rem}
Once a graph and a mass are fixed, the previous statements provides specific results for the defocusing case, provided that the associated homogeneous problem can be handled. In particular, by combining Proposition \ref{(un)bounded_defocusing} and Theorem \ref{defocusing_theorem} with the results of Section \ref{sec: homo} (which concern the homogeneous case), one obtains a fairly comprehensive picture of the defocusing regime for $1$ and $2$-periodic graphs. Here we consider two specific examples, for the sake of brevity. \\
1) Let $\cG$ be a $1$-periodic graph with a terminal point. Then, by \cite{periodic}, for every $p \in (2,6)$ and $\mu>0$ there exists a ground state for $\cE_p(\mu,\cG)$ on $\cG$. As a consequence, by Theorem \ref{alpha critical 1}-(2), there exists $\overline{\alpha}<0$ such that the same holds for $\cE_{p,q,\alpha}(\mu,\cG)$ for every $q \in (2,p)$ and $\alpha \in [\overline{\alpha},0)$, while for $\alpha<\overline{\alpha}$ there are no ground states. Moreover, if $p=6$ then the interval $(\mu_{6,\cG}, \tilde{\mu}_{\cG}]$ is empty (both values are equal to $\mu_{\R^+}$), and there are no ground state for any $\mu>0$ and $\alpha<0$.\\
2) Let $\cG$ be the square grid; the homogeneous problem is studied in \cite{square_grid}. If $p \in (2,4)$, for every $\mu>0$ there exists a ground state for $\cE_p(\mu,\cG)<0$ on $\cG$. Then by Theorem \ref{alpha critical 1}, there exists $\overline{\alpha}<0$ such that the same holds for $\cE_{p,q,\alpha}(\mu,\cG)$ for every $q \in (2,p)$ and $\alpha \in [\overline{\alpha},0)$, while for $\alpha<\overline{\alpha}$ there are no ground states. If $p \in [4,6)$, then there exists a ground state for $\cE_p(\mu,\cG)<0$ on $\cG$ only for $\mu>\mu_{p,\cG}$ (defined in \eqref{critical mass p}), while for $\mu \le \mu_{p,\cG}$ we have that $\cE_p(\mu\,,\cG)=0$. Therefore, by Theorem \ref{alpha critical 1}, when $\mu>\mu_{p,\cG}$ and $q \in (2,p)$ we can define $\overline{\alpha}<0$ as before, and we have existence of ground states for $\alpha \in [\overline{\alpha},0)$, and non-existence for $\alpha<\overline{\alpha}$. Instead, when $\mu \le \mu_{p,\cG}$, we have that $\overline{\alpha}=0$, and hence for any $\alpha<0$ there is no ground state for $\cE_{p,q,\alpha}(\mu,\cG)$. Finally, for $p=6$ it is known that $\cE_6(\mu,\cG)=0$ for $\mu \le \mu_{\R}$. Thus, for any such $\mu$ we still have that $\overline{\alpha}=0$, and the defocusing problem has never ground states.
\end{rem}

\begin{rem}
Theorem \ref{defocusing_theorem} is somehow of perturbative nature. In particular, the fact that $\cG$ is $1$ or $2$-periodic enters only in the analysis of the homogeneous problem. Actually, we stated Theorem \ref{defocusing_theorem} for $1$ and $2$-periodic graphs, for the sake of consistency with the rest of the paper. However, the result holds for general $d$-periodic graphs with $d \ge 1$, with the same proof. 

Clearly, as already observed, one needs to be able to analyze the homogeneous problem. For $d$-periodic graphs with $d \ge 3$, this is currently known only in the case of the three-dimensional cubic grid; see \cite{spatial_grid}.
\end{rem}


%

\noindent \textbf{Structure of the paper and notation.} In Section \ref{sec: pre}, we collect some preliminary results which will be frequently used in the rest of the paper. In Section \ref{sec: homo}, we state and prove old and new results on the homogeneous problem. Sections \ref{sec: foc} and \ref{sec: def} are devoted to the proofs of the main results in the focusing and defocusing cases, respectively.

In the following, working on a fixed metric graph $\mathcal{G}$, we will often omit the symbol $\mathcal{G}$ in the notations of $L^p(\mathcal{G})$, $H^1(\mathcal{G})$, $E_{p,q,\alpha}(\mu,\cG)$, $\mathcal{E}_{p,q,\alpha}(\mu, \mathcal{G})$ e.c., writing simply $L^p$, $H^1$, $E_{p,q,\alpha}(\mu)$, $\mathcal{E}_{p,q,\alpha}(\mu)$ e.c.; moreover, we will sometimes write $\|\cdot\|_p$, $\|\cdot\|_{H^1}$ instead of $\|\cdot\|_{L^p(\mathcal{G})}$, $\|\cdot\|_{H^1(\mathcal{G})}$.

\section{Preliminaries}\label{sec: pre}

We start by recalling the formal definition of \emph{periodic graph} from \cite[Section 4.1]{intro_graphs}, see also \cite{kuchment, periodic}.

\begin{definition}\label{def: periodic graph}
A metric graph $\mathcal{G}$, with set of vertexes $\mathcal{V}$ and set of edges $\mathcal{E}$, is said to be $d$-periodic, or $\mathbb{Z}^d$-periodic, if there is a group action of $\mathbb{Z}^d$ on $\mathcal{G}$,
\begin{equation}
\label{action}
\begin{split}
+:\mathbb{Z}^d\times \mathcal{G} &\longrightarrow \mathcal{G} \\
    (g,x) &\longmapsto g+x
    \end{split}
\end{equation}
such that:
\begin{itemize}
    \item $\forall g \in \mathbb{Z}^d$, the function $x \mapsto g+x$ is a graph automorphism, that is, it maps vertices into vertices, edges into edges, and it preserves the lengths;
    \item the action is free, namely $\forall g \in \mathbb{Z}^d$, $x \in \mathcal{G}$, if $g+x=x$ then $g=0$;
    \item the action is discrete, that is, $\forall x \in \mathcal{G}$, there is a neighborhood $U$ of $x$ such that $g+x \notin U$ for all $g \neq 0$;
    \item the action is co-compact, namely there exists a compact metric graph $K\subseteq \mathcal{G}$ such that $\mathcal{G}=\bigcup\limits_{g\in \mathbb{Z}^d}\left(g+K\right)$. We will call this subset $K$ a fundamental domain.
\end{itemize}
\end{definition}

\begin{rem}
In \cite{intro_graphs}, it is not required that a fundamental domain be itself a metric graph, but only that it be a compact subset of $\mathcal{G}$. A compact graph is a graph with finitely many edges and vertices, and such that each edge has finite length. It is clear that the requirement that the fundamental domain be a subgraph is not restrictive, under the additional assumption that the distance between any two vertices is bounded below by a positive quantity (see the ``finite ball condition" - Assumption 1.3.5 in \cite{intro_graphs}). 
\end{rem}

In this paper we focus on $1$-periodic and $2$-periodic metric graphs. As observed in \cite[Section 4.1]{intro_graphs}, we can always think that any such graph $\cG$ is embedded into $\R^3$, and the action is by
integer shifts with respect to vectors belonging to an isomorphic image
of the lattice $\mathbb{Z}$ or $\mathbb{Z}^2$. It is often convenient to adopt this point of view.

In what follows, we prove an elementary result which will be used in multiple steps.

\begin{lem}
\label{fund_dom}
    If $\mathcal{G}$ is a $d$-periodic metric graph, we can always choose a fundamental domain $K$ and an action $+$ in Definition \ref{def: periodic graph} in such a way that 
    \begin{equation}\label{cond int 2}
    \text{$K \cap (g+K)  = \emptyset$ for all $g \in \mathbb{Z}^d$ with $\|g\|_{\infty}\geq2$},
    \end{equation}
where $\|\cdot\|_\infty$ is the usual $\sup$ norm.
\end{lem}

\begin{proof}
Let $K$ be a fundamental domain, according to Definition \ref{def: periodic graph}; as a first step, we wish to show that, if necessary replacing $K$ with a larger subgraph of $\cG$, we can suppose that
\beq\label{1conn}
K \cap (g+K) \neq \emptyset \quad \forall g: \ \|g\|_\infty=1.
\eeq
If this is not the case, 
we fix $x_0 \in \mathcal{V}(K)$ and for any $g$ with $\|g\|_\infty=1$ we consider a path $\gamma_{x_0,g}$ connecting $x_0$ with $g+x_0$. Then we take
\[
K':= K \cup \bigcup_{\|g\|_\infty=1} \gamma_{x_0,g}.
\]
Since $g+x_0 \in \mathcal{V}(\cG)$, this is a subgraph of $\cG$, and by construction $K' \cap (g+K') \neq \emptyset \quad \forall g: \ \|g\|_\infty=1$. 
%
%

Now we rename $K'=K$, and let
\begin{equation*}
        s_K:=\min\left\{n\in\mathbb{N} | \ \text{ if } \|h\|_{\infty}\geq n \ \implies \  (h+K)\cap K = \emptyset  \right\};
\end{equation*}
this value indicates how large the $\infty$-norm of $h \in \mathbb{Z}^d$ needs to be in order for $(h+K)$ not to intersect $K$. We point out that $s_K$ is finite. Otherwise, there exists a sequence $\{g_n\}\subseteq \mathbb{Z}^d$ with $\|g_n\|_{\infty} \rightarrow +\infty$ such that $K \cap (g_n+K) \neq \emptyset$. This in turn would imply the existence of a sequence $\{x_n\} \subset K$ such that $-g_n+x_n \in K$ for every $n$. Since $K$ is compact, both $\{x_n\}$ and $\{-g_n+x_n\}$ converge, up to extracting subsequences. But, since $\|g_n\|_\infty \to \infty$ while $\{x_n\}$ is bounded, this is not possible.

Now, by \eqref{1conn} it is plain that $s_K\geq2$. If $s_K=2$, the proof is complete. Otherwise, the idea is to define a larger fundamental domain as follows: we set
    \begin{equation*}
        \tilde{K}=\bigcup\left\{g+K\text{ such that }
            \|g\|_{\infty}<s_K, \text{ with } g_i\geq 0\,\, \forall i \in \{1 \dots n\}\right\}
    \end{equation*}
    and consider the new group action of $\mathbb{Z}^d$ on $\mathcal{G}$ defined by
    \begin{equation*}
\begin{split}
+':\mathbb{Z}^d\times \mathcal{G} &\longrightarrow \mathcal{G} \\
    (g,x) &\longmapsto g+'x:=s_K g+x.
    \end{split}
\end{equation*}
This defines the same $d$-periodic graph $\mathcal{G}$, and we claim that:
\begin{itemize}
\item[($i$)] $\tilde{K}$ is a fundamental domain (with respect to the new action $+'$);
\item[($ii$)] condition \eqref{cond int 2} is satisfied when $\tilde K$ and $+'$ are used in place of $K$ and $+$.
\end{itemize}
Indeed, firstly
    \begin{equation*}
    \begin{split}
        &x \in \mathcal{G} \implies \exists h \in \mathbb{Z}^d, y \in K \text{ such that } \\
        &x=h+y = s_K g+t+y \text{ for some $g,t \in \mathbb{Z}^d$ with $\|t\|_{\infty}<s_K$ and  $t_i\geq 0$ for every $i$},
        \end{split}
    \end{equation*}
   having accomplished an Euclidean division on the components of the vector $h$ to obtain $g$ and $t$. This means that $\tilde{K}$ is a fundamental domain, since $t+y \in\tilde{K}$. \\
Furthermore, we can check that $s_{\tilde{K}} = 2$. Suppose that 
   \begin{equation*}
\emptyset\neq(h+'\tilde{K})\cap \tilde{K}=(s_K h+\tilde{K})\cap \tilde{K}.
   \end{equation*}
Then $\exists h_1, h_2 \in \mathbb{Z}^d$ with positive components such that $\|h_1\|_{\infty}, \|h_2\|_{\infty} <s_K$ and 
        \begin{equation*}
            (h_1+s_K h+K)\cap (h_2+K) \neq \emptyset \quad \iff \quad (h_1-h_2+s_Kh+K)\cap K \neq \emptyset,
        \end{equation*}
        whence we deduce that
    \begin{equation}
    \label{paggassi}
        s_K>\|h_1-h_2+s_K h\|_{\infty}\geq s_K \|h\|_{\infty}-\|h_1-h_2\|_{\infty} \geq s_K\|h\|_{\infty}-s_K
    \end{equation}
   (the last inequality holds since $\|h_1\|_{\infty}, \|h_2\|_{\infty} <s_K$ and their components are positive). Finally, \eqref{paggassi} implies that $\|h\|_{\infty}<2$, that is $s_{\tilde{K}}=2$, as claimed.
\end{proof}

\begin{rem}
We emphasize that the proof provides a constructive procedure to obtain a fundamental domain $K$ and an action $+$ satisfying the statement of the lemma, whenever this is not immediately evident.

For example, consider the two-periodic graph obtained by translating along the coordinate axes $x$ and $y$ the fundamental domain $K$ represented on the left-hand side of Figure~1. Recalling that one may regard a $2$-periodic graph as embedded in $\mathbb{R}^3$, in this case $K$ consists of the unit square together with a curved edge connecting $(0,1/2,0)$ to $(2,1/2,0)$. We may assume that this edge lies entirely in the half-space $\{z>0\}$, except for its endpoints, and that it does not intersect its translated copies.

If we took $K$ as a fundamental domain, with the action given by the usual translations by $(1,0,0) \simeq (1,0)$ and $(0,1,0) \simeq (0,1)$ (and their multiples), condition~\eqref{cond int 2} would fail, since $(2,1/2,0)\in K\cap ((2,0)+K)$. 

On the other hand, it is apparent that in this case $s_K=3$. We therefore define $\tilde K$ as in the proof. In this way we obtain the unit square grid restricted to $[0,3]\times [0,3]\times \{0\}$, together with nine curved edges connecting $(0,1/2,0)$ to $(2,1/2,0)$, $(1,1/2,0)$ to $(3,1/2,0)$, and so on. For the action, rather than using translations by multiples of $(1,0)$ and $(0,1)$, we consider the new action $+'\,$ given by translations by $(3,0)$ and $(0,3)$ (and multiples). By construction, $\tilde K$ and $+'$ may be used as a fundamental domain and action generating $\mathcal{G}$, and for them the conclusion of Lemma~\ref{fund_dom} holds.
\end{rem}

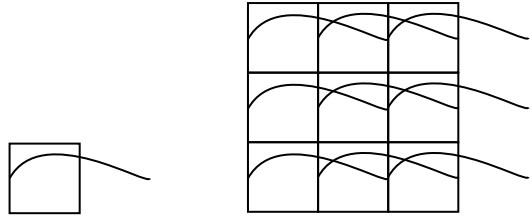
\begin{figure}[ht]
\begin{center}

\tikzset{every picture/.style={line width=0.75pt}} 

\begin{tikzpicture}[x=0.75pt,y=0.75pt,yscale=-0.7,xscale=0.7]

\draw   (81,141) -- (131,141) -- (131,191) -- (81,191) -- cycle ;
\draw    (81,166) .. controls (104.2,124.2) and (180.2,172.2) .. (181,166) ;
\draw   (251,40) -- (301,40) -- (301,90) -- (251,90) -- cycle ;
\draw    (251,66) .. controls (274.2,24.2) and (350.2,72.2) .. (351,66) ;
\draw   (301,40) -- (351,40) -- (351,90) -- (301,90) -- cycle ;
\draw    (301,65) .. controls (324.2,23.2) and (400.2,71.2) .. (401,65) ;
\draw   (351,40) -- (401,40) -- (401,90) -- (351,90) -- cycle ;
\draw    (351,65) .. controls (374.2,23.2) and (450.2,71.2) .. (451,65) ;
\draw   (251,90) -- (301,90) -- (301,140) -- (251,140) -- cycle ;
\draw    (251,116) .. controls (274.2,74.2) and (350.2,122.2) .. (351,116) ;
\draw   (301,90) -- (351,90) -- (351,140) -- (301,140) -- cycle ;
\draw    (301,115) .. controls (324.2,73.2) and (400.2,121.2) .. (401,115) ;
\draw   (351,90) -- (401,90) -- (401,140) -- (351,140) -- cycle ;
\draw    (351,115) .. controls (374.2,73.2) and (450.2,121.2) .. (451,115) ;
\draw   (251,140) -- (301,140) -- (301,190) -- (251,190) -- cycle ;
\draw    (251,166) .. controls (274.2,124.2) and (350.2,172.2) .. (351,166) ;
\draw   (301,140) -- (351,140) -- (351,190) -- (301,190) -- cycle ;
\draw    (301,165) .. controls (324.2,123.2) and (400.2,171.2) .. (401,165) ;
\draw   (351,140) -- (401,140) -- (401,190) -- (351,190) -- cycle ;
\draw    (351,165) .. controls (374.2,123.2) and (450.2,171.2) .. (451,165) ;

\end{tikzpicture}

\end{center}
\caption{\small On the left, the fundamental domain $K$, and on the right the new fundamental domain $\tilde K$.}
\label{fig:1}
\end{figure}

\subsection{Gagliardo-Nirenberg inequality and critical mass}

When dealing with minimization (and more in general with the search of critical points) for NLS-type energy, a key ingredient is represented by the Gagliardo-Nirenberg inequality. In the Euclidean space $\R^d$, it takes the form
\begin{equation}
\label{GN_abstract_R_d}
\|u\|^p_{L^p(\mathbb{R}^d)} \leq K_{p,d} \|u\|^{p-\alpha}_{L^2(\mathbb{R}^d)}\|\nabla u\|^{\alpha}_{L^2(\mathbb{R}^d)} \quad \text{where }\quad \alpha = \frac{d(p-2)}{2},
\end{equation}
for every $u \in H^1(\mathbb{R}^d)$, with $p \in [2, 2d/(d-2))$ if $d \ge 3$ and $p \in [2, +\infty)$ if $d=1,2$ (see \cite[Theorem 1.3.7]{Cazenave}).

On general noncompact metric graphs, the Gagliardo-Nirenberg inequality holds for $d=1$ (see e.g. \cite[Proposition 2.1]{threshold}): there exists an optimal constant $C_{q,\cG}>0$ depending on $q$ and $\cG$ such that
\beq\label{GN}
\lp{u}{q}^q \le  C_{q,\cG} \lp{u}{2}^{\frac{q+2}2} \lp{u'}2^{\frac{q-2}2} \qquad \forall u \in H^1(\cG).
\eeq
Precisely, $C_{q,\cG}$ is characterized as
\[
C_{q,\cG} = \sup_{u \in H^1(\cG) \setminus \{0\}} \frac{\lp{u}{q}^q}{\lp{u}{2}^{\frac{q+2}2} \lp{u'}2^{\frac{q-2}2}}.
\]
The inequality also holds for $q=\infty$: one obtains that
\beq\label{GN_infty}
\lp{u}{\infty} \le  C_{\infty,\cG} \lp{u}{2}^{\frac{1}2} \lp{u'}2^{\frac{1}2} \qquad \forall u \in H^1(\cG).
\eeq
A relevant role is played by the optimal constant obtained for the $L^2$-critical exponent $q=6$: one defines the \emph{critical mass} $\mu_{\cG}= \mu_{6,\cG}$ associated with $\cG$ as
\beq\label{critical mass}
\mu_{\cG} =\mu_{6,\cG}:= \sqrt{ \frac{3}{C_{6,\cG}} }.
\eeq
This value discriminates between existence and non-existence in the critical case, we refer in particular to \cite{critical}. Dealing with critical problems, it is also useful to recall the following \emph{modified Gagliardo-Nirenberg inequality}: 

\begin{lemma}[Lemma 4.4 in \cite{critical}]\label{lem: mod GN}
Assume that $\cG$ is non-compact and has no terminal point, and let $u \in H^1_\mu(\cG)$ for some $\mu \in (0,\mu_\R]$. Then there exists $\theta_u \in [0,\mu]$ such that
\[
\lp{u}{6}^6\le 3\left(\frac{\mu-\theta_u}{\mu_{\R}}\right)^2 \lp{u'}{2}^2 + C_{\cG} \theta_u^\frac12,
\]
with $C_{\cG} >0$ depending only on $\cG$.
\end{lemma}

The previous inequalities were proved for noncompact metric graphs with a finite number of edges, but the arguments of the proofs works also in the periodic setting. In particular, the proof of Lemma 4.4 in \cite{critical} relies on the existence of a path $\Gamma$ of infinite length which starts in a given point $x_0 \in \mathcal{G}$. This path is constructed connecting $x_0$ with the vertex of a half-line of $\mathcal{G}$. In our setting, $\mathcal{G}$ is $d$-periodic and does not contain any half-line, but, as already observed in \cite{periodic}, one can adapt the argument with minor changes.

\section{The homogeneous problem on periodic graphs}\label{sec: homo}

The purpose of this section is twofold. On the one hand, we recall the known results for $1$-periodic graphs. On the other hand, we extend and generalize the results in \cite{square_grid, honeycomb}, obtained for the square and honeycomb grids, to general $2$-periodic graphs. We recall that, for the homogeneous problem, we use the notation $E_p$ and $\cE_p$ introduced in \eqref{homog func}.

\subsection{$1$-periodic case} The investigation of the critical and subcritical regimes on $1$-periodic graphs is carried out in \cite{periodic}, leading to the following main results.
\begin{thm}[Theorem 1.1 in \cite{periodic}]
\label{1_periodic_subcritical}
    Let $\mathcal{G}$ be a $1$-periodic metric graph and $2<p<6$. Then for every $\mu>0$ we have that $\mathcal{E}_p(\mu, \mathcal{G})<0$ and there exists a ground state.
\end{thm}
The critical case, for $p=6$, is also studied in \cite{periodic}, and the following topological assumption on $\mathcal{G}$ plays a crucial role:
\[
\begin{split}
(H):& \text{ removing any edge $\mathrm{e}$ from $\cG$ generates only non-compact components}.
\end{split}
\]
As observed in \cite{periodic}, this in particular ensures that for every point $x \in \mathcal{G}$, there exist two simple curves starting at $x$ that are almost everywhere disjoint and of infinite length.

\begin{thm}[Theorems 1.2 and 1.3 in \cite{periodic}]
\label{1_periodic_critical}
If $\mathcal{G}$ is a $1$-periodic metric graph, then there exists a critical mass $\mu_{\mathcal{G}}$, defined in \eqref{critical mass}, such that: 
\begin{itemize}
\item[($i$)] when assumption $(H)$ holds, $\mu_{6,\mathcal{G}}=\mu_{\mathbb{R}}=\frac{\sqrt{3}\pi}{2}$ and
\begin{equation*}
\mathcal{E}_6(\mu, \mathcal{G})=
    \begin{cases}
        0 \quad &\text{ if } \mu\leq \mu_{\mathbb{R}} \\
        -\infty &\text{ if } \mu>\mu_{\mathbb{R}},
    \end{cases}
\end{equation*}
hence $\mathcal{E}_6(\mu, \mathcal{G})=\mathcal{E}_6(\mu, \mathbb{R})$ for every $\mu>0$. Moreover, the ground state level is never attained.
\item[($ii$)] If $\cG$ contains a terminal edge, then $\mu_{6,\mathcal{G}}=\mu_{\mathbb{R}^+}=\frac{\sqrt{3}\pi}{4}$, and
\begin{equation*}
    \mathcal{E}_6(\mu, \mathcal{G})=
    \begin{cases}
        0 \quad &\text{ if } \mu\leq \mu_{\mathbb{R}^+} \\
        -\infty &\text{ if } \mu>\mu_{\mathbb{R}^+},
    \end{cases}
\end{equation*}
hence $\mathcal{E}_6(\mu, \mathcal{G})=\mathcal{E}_6(\mu, \mathbb{R}^+)$ for every $\mu>0$. Moreover, the ground state level is never attained. 
\item[($iii$)] If $\mathcal{G}$ does not satisfy ($H$) and has no terminal edge, then
\begin{equation}
\label{stima massa critica}
    \mu_{\mathbb{R}^+} \leq \mu_{\mathcal{G}} \leq \mu_{\mathbb{R}}.
\end{equation}
If in addition $\mu_{\mathcal{G}}<\mu_{\mathbb{R}}$, then 
\begin{equation}
\label{valori inf energia}
    \mathcal{E}_6(\mu, \mathcal{G})=
    \begin{cases}
        0 \quad &\text{ if } \mu\leq \mu_{\mathcal{G}} \\
        <0 \quad &\text{ if } \mu_{\mathcal{G}}<\mu \leq \mu_{\mathbb{R}} \\
        -\infty &\text{ if } \mu>\mu_{\mathbb{R}},
    \end{cases}
    \end{equation}
and a ground state exists if and only if $\mu \in [\mu_{\mathcal{G}}, \mu_{\mathbb{R}}]$. 
\end{itemize}
\end{thm}
It is an open problem to determine if the equality $\mu_{\mathcal{G}}= \mu_{\mathbb{R}}$ can occur in this setting, and whether there is a ground state in this case (while cases in which $\mu_{\cG}<\mu_{\R}$ are known \cite{periodic}). 

\subsection{$2$-periodic case} 

When dealing with a $2$-periodic metric graph $\mathcal{G}$, the local $1$-dimensional nature of the single edges interacts with the global $2$-dimensional nature of the graph, due to the $2$-periodicity, leading to the so-called dimensional crossover. This was established for the square and the honeycomb grids in \cite{square_grid} and \cite{honeycomb}, respectively. 
In particular, in \cite{honeycomb} it has been observed that the strategy developed for the treatment of the square grid case \cite{square_grid} is rather flexible, and that essentially only two ingredients need to be adapted to more general settings: the proof of the two-dimensional Sobolev inequality (cf. \cite[Theorem 2.2]{square_grid}), and the construction of a suitable energy competitor, which allows one to show that the ground state energy level is strictly negative (cf. \cite[Proof of Theorem 1.1]{square_grid}). In what follows we establish the validity of these ingredients on general $2$-periodic graphs, as introduced at the beginning of Section \ref{sec: pre}, and we review the main steps in \cite{square_grid}. The starting point, as already mentioned is the following:

\begin{thm}[$2$-dimensional Sobolev inequality]
\label{SOBOLEV_2}
Let $\mathcal{G}$ be a $2$-periodic metric graph. Then, there exists a constant $C>0$ depending on $\cG$ such that
\begin{equation*}
    \|u\|_{L^2(\cG)} \leq C\|u'\|_{L^1(\cG)} \qquad \forall u \in W^{1,1}(\mathcal{G}).
\end{equation*}
\end{thm}
\begin{proof}
Let $K$ be the fundamental domain for $\mathcal{G}$, with associated action $+$, given by Lemma \ref{fund_dom}. We set
\begin{equation*}
\label{H_j,V_i}  H_j:=\bigcup\limits_{z\in\mathbb{Z}}((z,j)+K) \quad \quad \text{and} \quad V_i:=\bigcup\limits_{z\in\mathbb{Z}}((i,z)+K);
\end{equation*}
these can be considered horizontal and vertical strips.

For a fixed $x_0 \in K$, we define two Lipschitz paths $h$ and $v$, parametrized by arc length, with the following properties:
\begin{itemize}
\item[($i$)] $h$ and $v$ connect $x_0$ with $(1,0)+x_0$ and $(0,1)+x_0$, respectively;
\item[($ii$)] the images of both $h$ and $v$ cover $K$;
\item[($iii$)] both $h$ and $v$ have finite length, denoted by $\ell_h$ and $\ell_v$, respectively;
\item[($iv$)] $h([0,\ell_h]) \subset H_0$ and $v([0,\ell_v]) \subset V_0$.
\end{itemize}
We point out that $h$ and/or $v$ are not necessarily injective: self-intersections and multiple passes along certain edges are allowed. %
Further, we consider the curves obtained by gluing together the horizontal and vertical translations of $h$ and $v$:
\begin{equation*}
\begin{split}
h_j:\mathbb{R}&\longrightarrow H_j \\
t&\longmapsto h_j(t):=(m,j)+h(t-m\ell_h) \quad \text{if $t \in [m\ell_h, (m+1)\ell_h]$ for some $m \in \mathbb{Z}$},
\end{split}
\end{equation*}
and 
\begin{equation*}
\begin{split}
v_i:\mathbb{R}&\longrightarrow V_i \\
t&\longmapsto v_i(t):=(i,n)+h(t-n\ell_v) \quad \text{if $t \in [n \ell_v, (n+1)\ell_v]$ for some $n \in \mathbb{Z}$}.
\end{split}
\end{equation*}
We observe that $h_j$ and $v_i$ are well defined and cover $H_j$ and $V_i$,  respectively. 

Now, let $u \in W^{1,1}(\cG)$. The integrals of $u$ and $u'$ over $H_j$ and $V_i$ differ from the integrals of $u\circ h_j$ and $u\circ v_i$ over $\mathbb{R}$, since some edges can be crossed multiple times by $h_j$ and $v_i$. However, the following estimates hold for any $p \in [1, +\infty)$ and $j,i \in \mathbb{Z}$:
\begin{itemize}
    \item since $h_j$ (resp. $v_i$) crosses a certain edge of $H_j$ (resp. $V_i$) a maximum of $N$ times, for some $N\in\mathbb{N}$,
    \begin{equation*}
        \|u\circ h_j\|_{L^p(\mathbb{R})}\leq N\|u\|_{L^p(H_j)} \qquad \text{(resp.  }  \|u\circ v_i\|_{L^p(\mathbb{R})}\leq N\|u\|_{L^p(V_i)});
    \end{equation*}
    \item since the curves $h_j$ and $v_i$ are parametrized by arc length, the same holds for the integrals of the derivative:
    \begin{equation}
    \label{stima_derivate}
        \|(u\circ h_j)'\|_{L^p(\mathbb{R})}\leq N\|u'\|_{L^p(H_j)}\text{ and } \|(u\circ v_i)'\|_{L^p(\mathbb{R})}\leq N\|u'\|_{L^p(V_i)}.
    \end{equation}
\end{itemize}
In particular, $u\circ h_j, u\circ v_i \in W^{1,1}(\mathbb{R})$. Moreover, since $\mathcal{G}\subseteq \bigcup_{j\in\mathbb{Z}}H_j$,
we can assert that 
\begin{equation}
\label{decomposizione}
    \|u\|_{L^2(\cG)}^2\leq \sum\limits_{j \in \mathbb{Z}}\int_{H_j}|u|^2\,dx,
    \end{equation}
   %
and similarly, since $H_j \subset \bigcup_i (i,j)+K$, 
\begin{equation}
\label{decomp_al_contrario 2}
 \|u\|_{L^2(H_j)}^2 \le    \sum\limits_{i \in \mathbb{Z}}\int_{(i,j)+K}|u|^2\,dx.
\end{equation}
Furthermore, since $H_a \cap H_b = \emptyset$ if $|a-b|\geq 2$ (and a similar property holds for the vertical strips), we also have
\begin{equation}
\label{decomp_al_contrario}
    \sum\limits_{j \in \mathbb{Z}}\int_{H_j}|u'|\,dx \leq 2  \|u'\|_{L^1(\cG)}\quad \text{and} \quad \sum\limits_{i \in \mathbb{Z}}\int_{V_i}|u'|\,dx \leq 2  \|u'\|_{L^1(\cG)}.
\end{equation}

Let us consider now $j$ fixed, and take any point $x \in H_j$. Since $h$ covers $K$, there exists $t \in \mathbb{R}$ such that $h_j(t)=x$, and by the Fundamental Theorem of Calculus 
\begin{equation}
\label{TFC}
    u(x)-u(h_j(r))=\int_r^t (u\circ h_j)'(s)\,ds,
\end{equation}
for every $r<t$. Since $u \in C_0(\mathcal{G})$ (that is, $u$ is continuous and tends to $0$ at infinity), by taking the limit as $r \to - \infty$ in \eqref{TFC} we obtain 
\begin{equation}
\label{83}
|u(x)|\leq \int_{-\infty}^t |(u\circ h_j)'(s)|\,ds \leq \int_{\mathbb{R}} |(u\circ h_j)'(s)|\,ds\leq N\int_{H_j}|u'|\,d\mu = N\|u'\|_{L^1(H_j)},
\end{equation}
where we used \eqref{stima_derivate}. Now, there exists $i \in \mathbb{Z}$ such that $x \in (i,j)+K$, therefore $x$ lies also on the path $v_i$. With the same arguments that lead to \eqref{83}, we conclude that 
\begin{equation}
\label{84}
    |u(x)|\leq N\|u'\|_{L^1(V_i)}.
\end{equation}
    Multiplying \eqref{83} and \eqref{84} term by term, we obtain that for all $x \in (i,j)+K$,
\begin{equation*}
\label{85}
    |u(x)|^2\leq N^2\|u'\|_{L^1(H_j)}\|u'\|_{L^1(V_i)},
\end{equation*}
whence it follows that
\[
\int_{(i,j)+K} |u|^2\,dx \le N^2 \ell(K) \|u'\|_{L^1(H_j)}\|u'\|_{L^1(V_i)},
\]
where $\ell(K)$ denotes the total length of $K$. At this point we take the sum over $i$: by \eqref{decomp_al_contrario 2} and \eqref{decomp_al_contrario}, we infer that
\begin{equation*}
\begin{split}
    \int_{H_j}|u|^2\,dx \leq N^2\ell(K)\|u'\|_{L^1(H_j)}\left(\sum\limits_{i \in \mathbb{Z}}\|u'\|_{L^1(V_i)}\right) \leq 2N^2\ell(K) \|u'\|_{L^1(H_j)} \|u'\|_{L^1(\cG)};
\end{split}
\end{equation*}
By further summing over $j$, and recalling \eqref{decomposizione} and \eqref{decomp_al_contrario}, we finally obtain
\begin{equation}
\label{81}
    \|u\|_{L^2(\cG)}^2 \le  \sum_{j \in \mathbb{Z}} \int_{H_j}|u|^2\,dx \leq 2N^2 \ell(K)^2\left(\sum_{j \in \mathbb{Z}} \|u'\|_{L^1(H_j)}\right)\|u'\|_{L^1(\cG)} \leq 4N^2 \ell(K)^2\|u'\|_{L^1(\cG)}^2,
\end{equation}
which is the desired result.
\end{proof}

Building on this result, we are able to study the existence of ground states in full generality. As a first consequence of Theorem \ref{SOBOLEV_2}, arguing exactly as in \cite[Theorem 2.3 and Corollary 2.4]{square_grid} we deduce the following inter-dimensional Gagliardo-Nirenberg inequality: for every $p \in [4,6)$ there exists an optimal constant $K_{p,\cG}>0$ such that
\beq\label{int GN}
\|u\|_{L^p(\cG)}^p \le K_{p,\cG} \|u\|_{L^2(\cG)}^{p-2} \|u'\|_{L^2(\cG)}^2 \qquad \forall u \in H^1(\cG).
\eeq
In turn, this allows to define a continuum of critical masses as in \cite[Definition 4.1]{square_grid} by
\begin{equation}
\label{critical mass p}
\mu_p=\mu_{p, \mathcal{G}}:=\left(\frac{p}{2K_{p,\mathcal{G}}}\right)^{\frac{2}{p-2}} \qquad \text{for every $p \in [4,6)$}.
\end{equation}

Now, thanks to the $1$-dimensional Gagliardo-Nirenberg inequality, Eq. \eqref{GN}, the ground state energy level $\cE_p(\mu,\cG) >-\infty$ for every $p \in (2,6)$ and $\mu>0$. For $p=6$, this property holds only for masses below a critical threshold. In both cases, the lower boundedness of $E_p$ in $H^1_\mu$ opens the possibility that a ground state exists in this range. To establish existence, it is therefore necessary to analyze the behavior of minimizing sequences. Owing to the periodicity of the setting, translation invariance rules out loss of compactness due to escape to infinity. Consequently, the only remaining obstruction to convergence is vanishing (of the sequence and of the energy), which occurs in case of dispersion along the graph. It follows that, if a function with negative energy exists, any minimizing sequence must be convergent, and hence a ground state is attained. This euristic idea is made rigorous by the following statement: 

\begin{lem}
\label{negative}
If $\mathcal{G}$ is a $2$-periodic graph, $2<p<6$ and $\mathcal{E}_p(\mu, \mathcal{G})<0$ for some $\mu>0$, then there exists a ground state of mass $\mu$ for $E_p$.
\end{lem}
This proposition is proved in \cite[Proposition 3.1]{periodic} for $1$-periodic graphs, and in \cite[Proposition 3.3]{square_grid} for the square grid. The proof can be easily adapted step by step to general $2$-periodic metric graphs.

At this point we are in position to state and prove our results regarding the homogeneous case which generalize \cite{square_grid, honeycomb}. As in those cases, we have to distinguish between $p \in (2,4)$, $p \in [4,6)$ and $p=6$. 

\begin{thm}
\label{sottocritico_2_periodico}
    If $\mathcal{G}$ is a $2$-periodic metric graph and $p \in (2,4)$, then $-\infty<\mathcal{E}_p(\mu, \mathcal{G})<0$ and a ground state exists for every $\mu>0$.
\end{thm}


\begin{proof}
By the previous discussion, the thesis follows once that we find a function $u \in H^1_{\mu}(\cG)$ such that $E_p(u)<0$. Dealing with a general graph, this step does not follow directly from \cite{square_grid, honeycomb}. The existence of such a function can be obtained by proceeding as in the proof of Theorem \ref{sottocritico bidimensionale combined}, which will be given in the next section; in fact, the construction presented in that theorem can be viewed as a general case, and the present proof is recovered by choosing $\alpha=0$ therein. Thus, we refer to such proof for the details.
\end{proof}

\begin{rem}\label{rmk21}
We also obtain that $\cE_p(\mu, \cG) \le 0$ for every $\mu>0$ and $p \ge 4$ (see Corollary \ref{cor_combined} below).
\end{rem}

For $p \in [4,6]$, the dimensional crossover occurs: 
\begin{thm}
\label{teorema 2-sopracritico grafi 2-periodici}
Let $\mathcal{G}$ be a $2$-periodic metric graph. Then, for $p \in [4,6)$:
\begin{itemize}
\item[($i$)] if $\mu > \mu_{p,\mathcal{G}}$, then there exists a ground state, and $-\infty <\mathcal{E}_p(\mu)<0$;
\item[($ii$)] if $\mu < \mu_{p,\mathcal{G}}$, then there is no ground state and $\mathcal{E}_p(\mu)=0$;
\item[($iii$)] if $\mu=\mu_{p,\mathcal{G}}$, then $\mathcal{E}_p(\mu_{p,\mathcal{G}})=0$, and a ground state exists if $p \in (4,6)$.
\end{itemize}
\end{thm}
The problem of determining the existence of ground states for $p=4$ and $\mu=\mu_{4, \mathcal{G}}$ is open. 

\begin{proof}
The proof of this case is completely analogous to the one in \cite{square_grid}, and we only emphasize the key point: the interdimensional Gagliardo-Nirenberg inequality \eqref{int GN} gives the estimate
\[
E_p(u) \ge \frac12 \|u'\|_{L^2(\cG)}^2 \left(1-\left(\frac{\mu}{\mu_{p,\cG}}\right)^{\frac{p-2}{2}}\right) \qquad \forall u \in H^1_\mu(\cG).
\]
This, together with Remark \ref{rmk21}, implies that $\cE_p(\mu) = 0$ for $\mu \le \mu_{p,\cG}$, and that a ground state cannot exists when $\mu<\mu_{p,\cG}$. Namely, there exists an interval of masses for which the ground-state energy level is bounded from below, but not attained. Instead, for $\mu>\mu_{p,\cG}$, by considering a normalized maximizing sequence $\{v_n\}$ for the optimal value of the Gagliardo-Nirenberg constant $K_{p,\cG}$ in \eqref{int GN}, one obtains that $E_p(v_n)<0$ for sufficiently large $n$. Thus, Proposition \ref{negative} gives existence in such range of masses. We refer to \cite{square_grid} for the rest of the proof.
 \end{proof}

Finally, we consider $p=6$. In this framework, the techniques employed in proving Theorem \ref{1_periodic_critical} hold on $2$-periodic graphs and lead to an analogue result. We refer to \cite{square_grid} for the details.

\begin{thm}
\label{teorema critico grafi 2-periodici}
    Let $\mathcal{G}$ be a $2$-periodic metric graph. Then:
\begin{itemize}
    \item[($i$)] if $(H)$ is satisfied, then $\mu_{6,\mathcal{G}}=\mu_{\mathbb{R}}$, $\mathcal{E}_6(\mu, \mathcal{G})= \mathcal{E}_6(\mu, \mathbb{R})$ for all $\mu>0$, and the ground state level is never attained.
    \item[($ii$)] If $\mathcal{G}$ contains a terminal edge, then $\mu_{6,\mathcal{G}}=\mu_{\mathbb{R}^+}$, $\mathcal{E}_6(\mu, \mathcal{G})=\mathcal{E}_6(\mu, \mathbb{R}^+)$ for all $\mu>0$, and the ground state level is never attained.
    \item[($iii$)] If $\mathcal{G}$ violates $(H)$ and has no terminal edge, then \eqref{stima massa critica} and \eqref{valori inf energia} hold. Moreover, when $\mu_{6,\mathcal{G}}<\mu_{\mathbb{R}}$, the ground state exists if and only if $\mu \in [\mu_{\mathcal{G}}, \mu_{\mathbb{R}}]$. 
\end{itemize}
\end{thm}

Also for $2$-periodic graphs, it is currently unknown whether the equality $\mu_{6,\mathcal{G}}=\mu_{\mathbb{R}}$ can occur and whether ground states may exist in case ($iii$).

\section{The focusing case}\label{sec: foc}


In this section we consider the NLS energy with combined nonlinearity \eqref{def E} with $\alpha>0$. We adapt and suitably combine the strategies developed in \cite{square_grid, periodic} (homogeneous case on periodic graphs) and \cite{combined} (inhomogeneous case on noncompact graphs with a finite number of edges). As discussed in the introduction, the critical mass $\tilde{\mu}_{\mathcal{G}}$ defined in \eqref{critical mass combined} plays a key role for the lower boundedness of $E_{p,q,\alpha}$ in the critical case $p=6$.


\begin{proof}[Proof of Proposition \ref{(un)bounded}]
If $2<q<p<6$, we can use the Gagliardo-Nirenberg inequality \eqref{GN}, whence we deduce that
\begin{equation}\label{coe tip}
E_{p,q,\alpha}(u) \geq \frac{1}{2}\|u'\|_2^2-\frac{C}{p}\|u\|_2^{\frac{p}{2}+1}\|u'\|_2^{\frac{p}{2}-1}-\frac{\alpha}{q}C\|u\|_2^{\frac{q}{2}+1}\|u'\|_2^{\frac{q}{2}-1} \qquad \forall u \in H^1_\mu.
\end{equation}
This gives the lower boundedness and the coercivity of $E_{p,q,\alpha}$ in $H^1_{\mu}$.


Regarding the (un)boundedness in the case $p=6$, we distinguish between the cases when $\cG$ has a terminal point, or not. In the former case, the fact that $\cE_{p,q,\alpha}(\mu)=-\infty$ for every $\mu \ge \mu_{\R_+}=\tilde \mu_{\cG}$ can be proved as in \cite[Lemma 3.1]{combined}, without modifications. The lower boundedness and the coercivity for $\mu<\mu_{\R_+}$ follow directly from inequality \eqref{GN}. If instead $\cG$ has no terminal point, with $\mu <\mu_\R=\tilde \mu_\cG$, then the modified Gagliardo-Nirenberg inequality in Lemma \ref{lem: mod GN} gives
\begin{equation*}
\begin{split}
E_{6,q,\alpha}(u) &\geq \frac{1}{2}\|u'\|_2^2\left(1-\left(\frac{\mu-\theta_{u}}{\mu_{\mathbb{R}}}\right)^2\right)-\frac{C}{6}\theta_{u}^{\frac{1}{2}}-\frac{\alpha C}{q}\|u\|_2^{\frac{q}{2}+1}\|u'\|_2^{\frac{q}{2}-1} \\
& \geq \frac{1}{2}\|u'\|_2^2\left(1-\left(\frac{\mu}{\mu_{\mathbb{R}}}\right)^2\right)-\frac{C}{6}\mu^{\frac{1}{2}}-\frac{\alpha C}{q}\|u\|_2^{\frac{q}{2}+1}\|u'\|_2^{\frac{q}{2}-1}, 
\end{split}
\end{equation*}
for some $C>0$, which proves both coercivity and lower boundedness of $E_{6,q,\alpha}$ on $H^1_\mu$.

It remains to consider what happens for $\mu \ge \mu_\R=\tilde \mu_\cG$. If $\mu > \mu_\R$, it is sufficient to observe that $E_{p,q,\alpha}(u) \le E_{p,q,0}(u) = E_p(u)$, since $\alpha>0$. Therefore, taking the infimum over $u$, we deduce that
\[
\cE_{p,q,\alpha}(\mu) \le \cE_p(\mu) =-\infty \quad \forall \mu > \mu_\R.
\]
If $\mu=\mu_\R$, let $\phi$ be the even soliton for the problem on $\R$, and $\phi_\lambda(x):= \sqrt{\lambda}\phi(\lambda x)$. Let $\ell$ be the length of an edge $\mathrm{e} \in \cG$. We define
\[
u_\lambda(x) = \begin{cases} \frac{\mu_{\R}^{1/2}}{\|\phi_\lambda-\phi_\lambda(\ell)\|_{L^2(0,\ell)}} (\phi_\lambda(x)-\phi_\lambda(\ell)) & \text{if $x \in \mathrm{e}$} \\
0 & \text{if $x \in \cG \setminus \mathrm{e}$}.
\end{cases}
\]
With the same computations as in \cite[Lemma 3.1]{combined}, one obtains that $E_\alpha(u_\lambda,\cG) \to -\infty$ as $\lambda \to +\infty$.
\end{proof}

Key ingredients for the existence of ground states are represented by the following results, which are the natural counterparts of \cite[Lemma 3.2]{square_grid} and \cite[Proposition 3.3]{square_grid} (see also \cite[Proposition 3.1]{periodic}) in the present setting. The proofs follow the same arguments as in \cite{square_grid, periodic}, and hence are omitted.

\begin{lem}[Dichotomy]
\label{dich_combined}
Let $\mathcal{G}$ be a $1$ or $2$-periodic metric graph, and suppose that either $p<6$ and $\mu>0$, or $p=6$ and $\mu \in (0,\tilde{\mu}_{\mathcal{G}})$. Let also $\{u_n\}$ be a minimizing sequence for $E_{p,q,\alpha}$ in $H^1_{\mu}(\cG)$. Then $u_n\rightharpoonup u$ for some $u \in H^1(\cG)$, and either $u$ is a ground state of $E_{p,q,\alpha}$ in $H^1_{\mu}$, or $u \equiv 0$ and $\|u_n\|_{L^\infty(\cG)} \to 0$.
\end{lem}
\begin{proposition}
\label{negative_combined}
Let $\mathcal{G}$ be a $1$ or $2$-periodic metric graph, and suppose that either $p<6$ and $\mu>0$, or $p=6$ and $\mu \in (0,\tilde{\mu}_{\mathcal{G}})$. If $\mathcal{E}_{p,q,\alpha}(\mu)<0$, then there exists a ground state of mass $\mu$.
\end{proposition}

\subsection{On 1-periodic graphs}
To prove the existence of a ground state on $1$-periodic graphs, we show the existence of a function $u \in H^1_{\mu}(\cG)$ such that $E_{p,q,\alpha}(u)<0$, and then apply Proposition \ref{negative_combined}. The construction of the competitor relies on the special choice of the fundamental domain $K$ given by Lemma \ref{fund_dom} (in fact, for $1$-periodic graphs we could consider the sequence of functions $\{u_n\}$ constructed in \cite[Theorem 1.1]{periodic}; but we present here a different sequence, whose construction can be easily generalized in broader settings).

%
%

\begin{proof}[Proof of Theorem \ref{sottocritico_combined}]
We seek a function in $H^1_{\mu}$ with negative energy.  
Let $K$ be the fundamental domain of $\mathcal{G}$ obtained in Lemma \ref{fund_dom}. For all $n \in \mathbb{N}$, define
    \begin{equation*}
        B_n := \bigcup\limits_{|z|\leq n} \left(K+z\right).
    \end{equation*}
Letting $l(B_n)$ be the lenght of $B_n$, namely be the sum of the lengths of its edges, we note that $l(B_n) \asymp n$ for $n \rightarrow +\infty$, that is there exist $C_2\geq C_1>0$ such that, for $n$ sufficiently large,
\begin{equation}\label{l(B_n)}
    C_1 \leq \frac{l(B_n)}{n} \leq C_2.
\end{equation}
Indeed $l(B_n)\leq (2n+1)l(K)$, the last term being the sum of the lengths of all the copies of $K$ in $B_n$, and $l(B_n)\geq nl(K)$, as the translated copies $(g + K)$ are disjoint for $|g| \geq 2$, and thus at least $n$ of them fit in $B_n$ without overlapping.\\
Using the distance function on $\mathcal{G}$, that is, the shortest path distance, we define for all $n \in \mathbb{N}$:
    \begin{equation*}
    \begin{split}
d_n:\mathcal{G}&\longrightarrow\mathbb{R^+}\\
x &\longmapsto d(x,\overline{B_n^c}),
\end{split}
    \end{equation*}
the distance from $x$ to the topological closure of the complement of $B_n$. Observe that, since it is given by the distance function, $d_n$ is continuous and Lipschitz on $\mathcal{G}$, of Lipschitz constant $1$. 
For all $n,m \in \mathbb{N}$ such that $m-n$ is sufficiently large, we claim that
\begin{equation}
\label{maggiorazione}
    d(B_n, \overline{B_m^c}) \geq (m-n)C,
\end{equation}
for some constant $C>0$. Indeed, any curve $\gamma$ from $B_n$ to $B_m^c$ starts in $h_1+K$ and ends in $h_2+K$, with $|h_1|\leq n$ and $|h_2|\geq m+1$. Without loss of generality, let $h_1,h_2\geq 0$. Since $(g+K)\cap K \neq \emptyset$ only if $|g|\leq 1$, $\gamma$ has to cross all $t+K$ with $n\leq t\leq m+1$. Therefore, assuming that $m-n$ is even, we get:
    \begin{equation*} 
    \begin{split} 
        l(\gamma) &\geq d(n+K,n+2+K)+d(n+2+K,n+4+K)+\dots\\
        & \quad+d(m-2+K,m+K) =\frac{m-n}{2} d(K,2+K) \geq C(m-n),
        \end{split}
    \end{equation*} 
for $m-n$ sufficiently large. Observe that the invariance of the length of the edges $\mathrm{e}$ under the action $+$ was used. An analogue inequality is obtained when $m-n$ is odd or if $h_1$, $h_2$ are not both positive. \\
At this point, we define
    \begin{equation}
    \label{u_n}
        u_n(x) := \frac{d_{2n}(x)}{\max\limits_{x\in B_{2n}}\left(d_{2n}(x)\right)}\varepsilon_n
    \end{equation}
    for $n\in \mathbb{N}$, where the constant $\varepsilon_n$ is chosen so that $\|u_n\|_2^2=\mu$.\\
The following inequalities occur:
\begin{enumerate}
\item \label{1} $\forall x \in \mathcal{G}, |u_n(x)|\leq \varepsilon_n$, since $ \frac{d_{2n}(x)}{\max\limits_{x\in B_{2n}}\left(d_{2n}(x)\right)}\leq 1$;
\item \label{2}$|u_n'(x)|\leq \frac{\varepsilon_n}{\max\limits_{x\in B_{2n}}\left(d_{2n}(x)\right)}$ almost everywhere, since $|{d_{2n}}'(x)|\leq 1$, as discussed before;
\item \label{3}$\max\limits_{x\in B_{2n}}\left(d_{2n}(x)\right)\leq (2n+1)l(K)$, since any curve $\gamma$ starting in $x\in (g+K)\subseteq B_{2n}$ can reach $\overline{B_{2n}^c}$ passing through $h+K$ for all $|g|\leq |h| \leq |2n|$. There are less than $2n+1$ of these indexes $h$, therefore $l(\gamma)\leq (2n+1)l(K)$;
\item \label{4}$\max\limits_{x\in B_{2n}}\left(d_{2n}(x)\right)\geq d(B_0,\overline{B_{2n}^c}) \geq 2Cn$, by \eqref{maggiorazione};
\item \label{5}for all $x \in B_n$, $|u_n(x)| \geq \delta\varepsilon_n$ for some $\delta > 0$ independent of $n$, since 
\begin{equation*}
    |u_n(x)| \geq \frac{d(B_n, \overline{B_{2n}^c})}{\max\limits_{x\in B_{2n}}\left(d_{2n}(x)\right)}\varepsilon_n \geq \frac{Cn}{(2n+1)l(K)}\varepsilon_n\geq \delta \varepsilon_n
\end{equation*}
for some $\delta>0$ and $n$ sufficiently large. Inequality \eqref{maggiorazione} and point \ref{3} were used.
\end{enumerate}
From these, we estimate:
\begin{itemize}
    \item $\forall \,\,2\leq p \leq \infty, \|u_n\|_p^p\asymp \varepsilon_n^p n$, since
    \begin{equation*}
        \|u_n\|_p^p\leq \varepsilon_n^p l(B_{2n})\leq C\varepsilon_n^pn,
    \end{equation*}
    having used \eqref{l(B_n)} and point \ref{1}, and
    \begin{equation*}
        \|u_n\|_p^p\geq \delta^p\varepsilon_n^pl(B_n)\geq C\varepsilon_n^pn,
    \end{equation*}
    for \eqref{l(B_n)} and point \ref{5}.
    \item $\varepsilon_n\asymp \frac{1}{\sqrt{n}}$, since $\varepsilon_n$ is chosen so that $\mu=\|u_n\|_2^2\asymp \varepsilon_n^2n$.
    \item $\|u_n'\|_2^2\leq \frac{C}{n^2}$ for some constant $C>0$, since 
    \begin{equation*}
        \|u_n'\|_2^2 \leq \frac{\varepsilon_n^2}{\left(\max\limits_{x\in B_{2n}}\left(d_{2n}(x)\right)\right)^2}l(B_{2n}) \leq \frac{C}{n^2},
    \end{equation*}
    having used \eqref{l(B_n)} and points \ref{2} and \ref{4}.
\end{itemize}
Therefore, for some constants $C_1, C_2, C_3>0$,
\begin{equation}
\label{108}
    E_{p,q,\alpha}(u_n)\leq \frac{C_1}{n^2}-\frac{C_2}{n^{\frac{p}{2}-1}}-\frac{\alpha C_3}{n^{\frac{q}{2}-1}},
\end{equation}
and as $q < p \leq 6$, the right-hand side of inequality \eqref{108} is negative for $n$ sufficiently large. Thus, Lemma \ref{negative_combined} gives the thesis.
\end{proof}

\subsection{On 2-periodic graphs}

Regarding $2$-periodic graphs, we still use Proposition \ref{negative_combined}, however distinguishing the different ranges of exponents arising from the dimensional crossover. 

\medskip

\noindent \textbf{The case $q \in (2,4)$.} This is the actual subcritical case, in the sense that $q$ is below the critical exponents of both dimensions $1$ and $2$. Exploiting this property, we are able to prove the negativity of $\cE_{p,q,\alpha}(\mu)$ independently of $\mu$, $\alpha$ and of the relation between $p$ and $q$.


\begin{proof}[Proof of Theorem \ref{sottocritico bidimensionale combined}]
We make use of a $2$-periodic adaptation of the sequence $\{u_n\}$ defined in Theorem~\ref{sottocritico_combined}. Again, we choose a fundamental domain $K$ that satisfies the requests of Lemma \ref{fund_dom} and we set
\begin{equation*}
\label{H_j,V_i}  H_j:=\bigcup\limits_{z\in\mathbb{Z}}((z,j)+K), \quad V_i:=\bigcup\limits_{z\in\mathbb{Z}}((i,z)+K), \quad
B_n := \bigcup\limits_{|z|\leq n} \left(K+z\right).
\end{equation*}
    We have that
    \begin{equation}
    \label{2-dim l(B_n)}
           l(B_n) \asymp n^2,
    \end{equation}
    since $l(B_n)\leq (2n+1)^2l(K)$, which is the sum of the lengths of all the copies of $K$ in $B_n$, and $l(B_n)\geq n^2l(K)$, since in $B_n$ there are at least $n^2$ disjoint copies of $K$, under the assumption that $(g+K)\cap K =\emptyset$ if $\|g\|_{\infty}\geq 2$. Moreover, for all $n,m \in \mathbb{N}$ such that $m-n$ is sufficiently large,
    \begin{equation}
    \label{2-dim maggiorazione}
        d(B_n, \overline{B_m^c}) \geq (m-n)C,
    \end{equation}
    for some $C>0$. Indeed, any curve $\gamma$ from $B_n$ to $B_m^c$ starts in $g+K$ and ends in $h+K$, with $\|g\|_{\infty}\leq n$ and $\|h\|_{\infty}\geq m+1$. Let $g=(g_1,g_2)$ and $h=(h_1,h_2)$, without loss of generality we can assume that $h_2\geq m+1$. Hence, $\gamma$ starts in $H_{g_2}$ and ends in $H_{h_2}$, and since $(t+K)\cap K \neq \emptyset$ only if $\|t\|_{\infty}\leq 1$, the curve $\gamma$ must cross all $H_i$ with $g_2 \leq i \leq h_2$. Therefore, assuming that both $g_2$ and $h_2$ are even or odd,
    \begin{equation*}
    \begin{split}
        l(\gamma) & \geq  d(H_{g_2},H_{g_2+2})+d(H_{g_2+2},H_{g_2+4})+\dots\\
        & \quad
       +d(H_{h_2-2},H_{h_2}) = \frac{h_2-g_2}{2} d(H_0,H_2) \geq C(m-n),
        \end{split}
    \end{equation*} 
for $m-n$ sufficiently large. An analogue estimate holds without the assumption that both $g_2$ and $h_2$ are even or odd.\\
We define now the sequence $\{u_n\}$ exactly as in \eqref{u_n}, namely
\begin{equation*}
u_n(x) := \frac{d_{2n}(x)}{\max\limits_{x\in B_{2n}}\left(d_{2n}(x)\right)}\varepsilon_n,
\end{equation*} 
with $d_n(x):=d(x,\overline{B_n^c})$ and $\varepsilon_n>0$ being such that $\|u_n\|_2^2=\mu$. Following the same arguments of points 1, 2, 3, 4 and 5 in Theorem \ref{sottocritico_combined}, we obtain analogous estimates for $\varepsilon_n$ and $E_{p,q,\alpha}(u_n)$. The exponents of these approximations differ from what we obtained for the $1$-periodic case, due to \eqref{2-dim l(B_n)} and \eqref{2-dim maggiorazione}:
\begin{itemize}
    \item $\|u_n\|_p^p\asymp \varepsilon_n^pn^2$, since $|u_n| \leq \varepsilon_n$ in $B_{2n}$, $|u_n| \geq \delta\varepsilon_n$ in $B_n$ for some $\delta>0$;
    \item $\varepsilon_n\asymp \frac{1}{n}$, since $\mu=\|u\|_2^2\asymp \varepsilon_n^2n^2$;
    \item $\|u'_n\|_2^2 \leq \frac{\varepsilon_n^2}{\left(\max\limits_{x\in B_{2n}}\left(d_{2n}(x)\right)\right)^2}l(B_{2n}) \leq \frac{\varepsilon_n^2}{\left(d(B_0,\overline{B_{2n}^c})\right)^2}l(B_{2n}) \leq \frac{C}{n^2}$, for some $C>0$, by \eqref{2-dim l(B_n)} and \eqref{2-dim maggiorazione}.
\end{itemize}
Therefore, there exist some constants $C_1, C_2, C_3>0$ such that:
\begin{equation}\label{st2912}
    E_{p,q,\alpha}(u_n)\leq \frac{C_1}{n^2}-\frac{C_2}{n^{p-2}}-\frac{\alpha C_3}{n^{q-2}}.
\end{equation}
Since $q<4$, this quantity is negative for $n$ sufficiently large. The thesis is obtained from Lemma \ref{negative_combined}.
\end{proof}

Observe that the same argument applies for any choice of $q$, $p$, $\mu$, $\alpha$; in particular, as already mentioned in the proof of Theorem \ref{sottocritico_2_periodico}, it works for $p \in (2,4)$ and $\alpha=0$. Moreover, estimate \eqref{st2912} gives the following:
\begin{cor}
\label{cor_combined}
    If $\mathcal{G}$ is a $2$-periodic metric graph, then $\mathcal{E}_{p,q,\alpha}(\mu)\leq 0$ for every $\mu>0$, $\alpha \in \R$ and $4 \leq q < p \leq 6$.
\end{cor}

\medskip

\noindent \textbf{The case $4 \leq q < p \leq 6$.} In this range, the so-called dimensional crossover occurs and the existence of the ground state can be established only when the mass exceeds a critical value. It is convenient to recall the definition of the critical masses $\mu_{p,\cG}$ given in \eqref{critical mass} and \eqref{critical mass p}, which throughout this section we will simply denote by $\mu_p$. We write the energy functional $E_{p,q,\alpha}$ on $H^1_{\mu}(\cG)$ as
\begin{equation*}
    E_{p,q,\alpha}(u)= \frac{1}{2}\|u'\|_2^2\left(1-\frac{2}{p}Q_p(u)\mu^{\frac{p-2}{2}}-\frac{2\alpha}{q}Q_q(u)\mu^{\frac{q-2}{2}}\right),
\end{equation*}
where 
\beq\label{def quoz GN}
Q_r(u):=\frac{\|u\|_r^r}{\|u'\|_2^2\|u\|_2^{r-2}}.
\eeq
We then define
\begin{equation}
\label{decomp_combined}
    F_{p,q,\alpha}(u, \mu):=1-\frac{2}{p}Q_p(u)\mu^{\frac{p-2}{2}}-\frac{2\alpha}{q}Q_q(u)\mu^{\frac{q-2}{2}},
\end{equation}
and let 
\begin{equation*}
\mathcal{F}_{p,q,\alpha}(\mu):=\inf\limits_{u \in H^1_{\mu}(\cG)}F_{p,q,\alpha}(u, \mu)=\inf\limits_{u \in H^1(\cG)\setminus \{0\}}F_{p,q,\alpha}(u, \mu),
\end{equation*}
where the last equality is due to the invariance of $F_{p,q,\alpha}$ under the transformation of $u$ in $\lambda u$, for $\lambda \in\mathbb{R}$. \\
We prove the existence of a unique critical mass $\mu_{p,q,\alpha}=\mu_{p,q,\alpha,\cG} \in (0,+\infty)$ such that:
\begin{equation}
\label{massa critica}
    \begin{cases}
\mathcal{F}_{p,q,\alpha}(\mu)>0 \quad &\text{ if } \mu<\mu_{p,q,\alpha} \\
        \mathcal{F}_{p,q,\alpha}(\mu)=0 \quad &\text{ if } \mu=\mu_{p,q,\alpha} \\
        \mathcal{F}_{p,q,\alpha}(\mu)<0 \quad &\text{ if } \mu>\mu_{p,q,\alpha}. \\
    \end{cases}
\end{equation} 

We make use of the following lemma, analogous to \cite[Lemma 3.2]{combined}. The proof can be repeated essentially step by step, and hence is omitted.

\begin{lem}
\label{continua}
    For $p<6$, the function $\mathcal{F}_{p,q,\alpha}:[0, +\infty)\rightarrow \mathbb{R}$, extended as $1$ in $\mu=0$, is continuous. The same happens for $p=6$, for $\mathcal{F}_{6,q,\alpha}:[0, \tilde{\mu}_{\mathcal{G}})\rightarrow \mathbb{R}$.
\end{lem}

This continuity result can be used to prove the existence of a zero for the function $\mu \mapsto \mathcal{F}_{p,q,\alpha}(\mu)$, using the Intermediate Value Theorem. Observe that 
\begin{equation*}
F_{p,q,\alpha}(u, \mu)=1-\frac{2}{p}Q_p(u)\mu^{\frac{p-2}{2}}-\frac{2\alpha}{q}Q_q(u)\mu^{\frac{q-2}{2}}<1 \qquad \forall u \in H^1_{\mu}(\cG),
\end{equation*}
whence 
\beq\label{30121}
\mathcal{F}_{p,q,\alpha}(\mu)<1 \qquad \text{for every $\mu$}.
\eeq 
We now prove a result related to the monotonicity of the function $\mathcal{F}_{p,q,\alpha}$.
\begin{lem}
\label{decreasing}
    If $p<6$, $\mu>0$ and $\theta>1$, then 
    \begin{equation}
    \label{8.9}
        \mathcal{F}_{p,q,\alpha}(\theta\mu)-1<\theta\left(\mathcal{F}_{p,q,\alpha}(\mu)-1\right).
    \end{equation}
    The same happens if $p=6$ and $\theta\mu<\tilde{\mu}_{\mathcal{G}}$.
\end{lem}
\begin{proof}
Let $\{u_n\}\subseteq H^1_{\mu}$ be a minimizing sequence for $\mathcal{F}_{p,q,\alpha}(\mu)$. Then $\theta^{\frac{1}{2}}u_n \in H^1_{\theta\mu}(\cG)$ is such that 
\begin{equation*}
\begin{split}
    \mathcal{F}_{p,q,\alpha}(\theta\mu) &\leq F_{p,q,\alpha}(\theta^{\frac{1}{2}} u_n, \theta \mu) = 1-\frac{2}{p}Q_p(u_n)\mu^{\frac{p-2}{2}}\theta^{\frac{p-2}{2}}-\frac{2\alpha}{q}Q_q(u_n)\mu^{\frac{q-2}{2}}\theta^{\frac{q-2}{2}} \\
    &< 1-\frac{2}{p}Q_p(u_n)\mu^{\frac{p-2}{2}}\theta-\frac{2\alpha}{q}Q_q(u_n)\mu^{\frac{q-2}{2}}\theta=1+\theta\left(F_{p,q,\alpha}(u_n, \mu)-1\right),
\end{split}
\end{equation*}
where the second inequality holds as $p,q \in [4,6]$. The thesis follows by taking the limit as $n \to \infty$.
\end{proof}

We are ready to prove the existence of a unique critical mass.
\begin{lem}
\label{massa critica esistenza}
    There exists a unique critical mass $\mu_{p,q,\alpha} \in (0, \mu_{p})$ satisfying \eqref{massa critica}. 
%
\end{lem}
\begin{proof}
Lemmas \ref{continua} and \ref{decreasing} imply that the function $\mathcal{F}_{p,q,\alpha}$ is continuous and strictly decreasing in $\mu$: indeed, if $\mu_1<\mu_2$ (and in addition $\mu_2<\tilde \mu_{\cG}$ in case $p=6$), applying \eqref{8.9} with $\theta = \mu_2/\mu_1>1$ we obtain
\begin{equation*}
    \mathcal{F}_{p,q,\alpha}(\mu_2)-1<\frac{\mu_2}{\mu_1}\left(\mathcal{F}_{p,q,\alpha}(\mu_1)-1\right)<\mathcal{F}_{p,q,\alpha}(\mu_1)-1,
\end{equation*}
thanks to \eqref{30121}. Thus, the thesis follows from the Intermediate Values Theorem, once that we show that $\mathcal{F}_{p,q,\alpha}(\mu_{p})<0$. To prove this fact, we observe at first that by definition $\mathcal{F}_{p,q,\alpha}(\mu_{p})<0$ if and only if $\mathcal{E}_{p,q,\alpha}(\mu_{p})<0$, and hence we can equivalently prove this inequality.

Now, let $p \in (4,6)$. In this case, by Theorem \ref{teorema 2-sopracritico grafi 2-periodici} there exists $u_p \in H^1_{\mu_{p}}(\mathcal{G})$ such that $E_p(u_p) = 0$. As a consequence, for every $\alpha>0$ and $q \in (4,6)$ we have that $\mathcal{E}_{p,q,\alpha}(\mu_{p}) \le E_{p,q,\alpha}(u_p)<0$,
as desired.

The very same argument also works for $p=6$ in case $\cG$ does not fulfill assumption ($H$), has no terminal point, and $\mu_{\cG}<\mu_{\R}$. Indeed, in this setting there exists a ground state for $\mathcal{E}_{6,q,\alpha}(\mu_{6}) =0$, by Theorem \ref{teorema critico grafi 2-periodici}-($iii$).

It remains to discuss what happens for $p=6$ if either assumption ($H$) holds, or $\cG$ has a terminal point, or else $\cG$ does not fulfill assumption ($H$), has no terminal point, and $\mu_{\cG} = \mu_{\R}$. In any of these cases, we have that $\mu_{6} = \mu_{\cG}$ coincides with $\tilde \mu_{\cG}$. Thus, the fact that $\mathcal{E}_{p,q,\alpha}(\mu_{p})<0$ follows directly from Proposition \ref{(un)bounded}.
\end{proof}


The previous results allow us to proceed with the:

\begin{proof}[Proof of Theorem \ref{teorema 2 sopracritico combined}]
The proof follows the same strategy as in the homogeneous case.

Let $\mu<\mu_{p,q,\alpha}$. By \eqref{massa critica}, we have that
    $\mathcal{F}_{p,q,\alpha}(\mu)> 0$, and hence
     \begin{equation*}
         E_{p,q,\alpha}(u) = \frac{1}{2}\|u'\|_2^2F_{p,q,\alpha}(u,\mu)>0 \qquad \forall u \in H^1_\mu(\cG).
\end{equation*}
But by Corollary \ref{cor_combined} we also have that $\mathcal{E}_{p,q,\alpha}(\mu)\leq 0$. Therefore, $\mathcal{E}_{p,q,\alpha}(\mu) = 0$, and the ground state energy level is not reached. 

Let now $\mu >\mu_{p,q,\alpha}$. By Proposition \ref{(un)bounded} and Lemma \ref{massa critica esistenza}, we have that $\mathcal{E}_{p,q,\alpha}(\mu) \in (-\infty,0)$. Therefore, the thesis follows from Proposition \ref{negative_combined}.

It remains to discuss the delicate case $\mu=\mu_{p,q,\alpha}$. By \eqref{massa critica}, we have that
    \begin{equation*}
        E_{p,q,\alpha}(u)=\frac{1}{2}\|u'\|_2^2F_{p,q,\alpha}(u, \mu_{p,q,\alpha})\geq 0 \qquad \forall u \in H^1_{\mu_{p,q,\alpha}}(\cG).
    \end{equation*}
    This, combined with Corollary \ref{cor_combined}, gives $ \mathcal{E}_{p,q,\alpha}(\mu_{p,q,\alpha})=0$. To obtain a ground state in case $q>4$, we consider a minimizing sequence $\{u_n\} \subseteq H^1_{\mu_{p,q,\alpha}}(\cG)$ for $F_{p,q,\alpha}(\cdot, \mu_{p,q,\alpha})$. This is also a bounded minimizing sequence for $E_{p,q,\alpha}$ constrained in $H^1_{\mu_{p,q,\alpha}}(\cG)$. To prove this claim, we discuss separately the cases $p<6$ and $p=6$. In the former one, the Gagliardo-Nirenberg inequality \eqref{GN} yields
\begin{equation*}
\begin{split}
    F_{p,q,\alpha}(u_n, \mu_{p,q,\alpha})&=1-\frac{2}{p}\frac{\|u_n\|_p^p}{\|u_n'\|_2^2}-\frac{2\alpha}{q}\frac{\|u_n\|_q^q}{\|u_n'\|_2^2} \geq 1-\frac{2}{p}C\frac{\|u_n\|_2^{\frac{p}{2}+1}}{\|u_n'\|_2^{3-\frac{p}{2}}}-\frac{2\alpha}{q}C\frac{\|u_n\|_2^{\frac{q}{2}+1}}{\|u_n'\|_2^{3-\frac{q}{2}}}\\
    & = 1-\frac{2}{p}C\frac{{\mu_{p,q,\alpha}}^{\frac{p+2}{4}}}{\|u_n'\|_2^{3-\frac{p}{2}}}-\frac{2\alpha}{q}C\frac{{\mu_{p,q,\alpha}}^{\frac{q+2}{4}}}{\|u_n'\|_2^{3-\frac{q}{2}}},
\end{split}
\end{equation*}
for some $C>0$. If $\{\|u_n'\|_2\}$ is not bounded, then
\begin{equation*}
 0=   \mathcal{F}_{p,q,\alpha}(\mu_{p,q,\alpha})= \lim\limits_n F_{p,q,\alpha}(u_n, \mu_{p,q,\alpha})= 1,
\end{equation*}
a contradiction. Therefore, $\{u_n\}$ is bounded in $H^1_{\mu_{p,q,\alpha}}$, and 
\begin{equation*}
E_{p,q,\alpha}(u_n)= \frac{1}{2}\|u_n'\|_2^2F_{p,q,\alpha}(u_n)\rightarrow 0,    
\end{equation*}
which proves the claim. If instead $p=6$, then we apply the Gagliardo-Nirenberg inequality \eqref{GN} and obtain
\begin{equation*}
     F_{6,q,\alpha}(u_n, \mu_{p,q,\alpha})=1-\frac{1}{3}\frac{\|u_n\|_6^6}{\|u_n'\|_2^2}-\frac{2\alpha}{q}\frac{\|u_n\|_q^q}{\|u_n'\|_2^2} \geq 1-\frac{1}{3}K_{\mathcal{G}}{\mu_{6,q,\alpha}}^2-\frac{2\alpha}{q}C\frac{{\mu_{6,q,\alpha}}^{\frac{q+2}{4}}}{\|u_n'\|_2^{3-\frac{q}{2}}}.
\end{equation*}
As before, if $\{\|u_n'\|_2\}$ is not bounded, then
\begin{equation}
\label{8.11}
    \mathcal{F}_{6,q,\alpha}(\mu_{6,q,\alpha})\geq 1-\frac{1}{3}K_{\mathcal{G}}\mu_{6,q,\alpha}^2>1-\frac{1}{3}K_{\mathcal{G}}{\mu_{\mathcal{G}}}^2,
\end{equation}
the last inequality following from the fact that $\mu_{6,q,\alpha}<\mu_{6}=\mu_{\mathcal{G}}$, as shown in Lemma \ref{massa critica esistenza}. By the definition of $\mu_{\mathcal{G}}$, the last term of \eqref{8.11} vanishes, yielding a contradiction with \eqref{massa critica}.
Therefore, $\{u_n\}$ is bounded in $H^1_{\mu_{p,q,\alpha}}(\cG)$, and we deduce the validity of the claim as before.

At this point, let $\{u_n\}$ be the previous bounded minimizing sequence. Up to a subsequence $u_n\rightharpoonup u$ for some $u \in H^1(\cG)$, and either $u\equiv 0$, or $u \in H^1_{\mu_{p,q,\alpha}}(\cG)$ is a ground state, by Lemma \ref{dich_combined}. Moreover, in the former case we have that $\|u_n\|_{L^\infty(\cG)} \to 0$. To complete the proof, we assume that $u \equiv 0$ and seek for a contradiction. If $q>4$, then
\begin{equation*}
\begin{split}
    0 & = \lim_n F_{p,q,\alpha}(u_n,\mu_{p,q,\alpha})=\lim_n \left(1-\frac{2}{p}\frac{\|u_n\|_p^p}{\|u_n'\|_2^2}-\frac{2\alpha}{q}\frac{\|u_n\|_q^q}{\|u_n'\|_2^2}\right) \\
    & \geq \lim_n \left( 1-\|u_n\|_{\infty}^{p-4}\frac{2}{p}\frac{\|u_n\|_4^4}{\|u_n'\|_2^2\|u_n\|_2^2}{\mu_{p,q,\alpha}}-\|u_n\|_{\infty}^{q-4}\frac{2\alpha}{q}\frac{\|u_n\|_4^4}{\|u_n'\|_2^2\|u_n\|_2^2}{\mu_{p,q,\alpha}}\right) \\
    &  \geq \lim_n \left(1-\|u_n\|_{\infty}^{p-4}\frac{2}{p}K_{4,\mathcal{G}}{\mu_{p,q,\alpha}}-\|u_n\|_{\infty}^{q-4}\frac{2\alpha}{q}K_{4,\mathcal{G}}{\mu_{p,q,\alpha}}\right) = 1,\\
\end{split}
\end{equation*}
which is not possible. 
\end{proof}

Finally:

\begin{proof}[Proof of Theorem \ref{int alpha}]
The estimate $\mu_{p,q,\alpha}<\mu_p$ follows from Lemma \ref{massa critica esistenza}. Regarding the other estimate from above, we start from the case $q \in (4,6)$. 
We claim that 
\begin{equation*}
    \mathcal{F}_{p,q,\alpha}\left(\left(\frac{1}{\alpha}\right)^{\frac{2}{q-2}} \mu_{q}\right)<0.
\end{equation*}
This implies that $\mu_{p,q,\alpha} < \mu_q / (\alpha^{2/(q-2)})$, recalling \eqref{massa critica}. By Theorem \ref{teorema 2-sopracritico grafi 2-periodici}, there exists $u_q \in H^1_{\mu_{q}}(\cG)$ such that $E_q(u_q) =0$. By definition \eqref{critical mass p}, this gives that $Q_q(u_q)=K_{q,\mathcal{G}}$, and moreover
\begin{equation*}
    \begin{split}
\mathcal{F}_{p,q,\alpha}\left(\left(\frac{1}{\alpha}\right)^{\frac{2}{q-2}} \mu_{q}\right) &\leq F_{p,q,\alpha}\left(\left(\frac{1}{\alpha}\right)^{\frac{1}{q-2}} u_{q}, \left(\frac{1}{\alpha}\right)^{\frac{2}{q-2}} \mu_{q}\right)\\
&=1-\frac{2}{p}Q_p(u_q){\mu_{q}}^{\frac{p-2}{2}}\left(\frac{1}{\alpha}\right)^{\frac{p-2}{q-2}}-\frac{2\alpha}{q}Q_q(u_q){\mu_{q}}^{\frac{q-2}{2}}\frac{1}{\alpha}\\
&< 1-\frac{2}{q} K_{q,\mathcal{G}}{\mu_{q}}^{\frac{q-2}{2}} =0,
    \end{split}
\end{equation*}
as desired. When $q=4$, the existence of a ground state of mass $\mu_{4}$ is not guaranteed. However, by optimality of $K_{4,\cG}$ in \eqref{int GN},
\begin{equation*}
    \forall \varepsilon >0 \  \exists u_{\varepsilon} \in H^1_{\mu_{4}}(\cG) \text{ such that } Q_4(u_{\varepsilon})>K_{4, \mathcal{G}}-\varepsilon.
\end{equation*}
Therefore, 
\begin{equation*}
\begin{split}
\mathcal{F}_{p,4,\alpha}\left(\frac{\mu_{4}}{\alpha}\right) &\leq F_{p,4,\alpha}\left(\frac{u_\varepsilon}{\sqrt{\alpha}}, \frac{\mu_{4}}{\alpha}\right) \\
&= 1-\frac{2}{p}Q_p(u_{\varepsilon})\left(\frac{\mu_{4}}{\alpha}\right)^{\frac{p-2}{2}}-\frac{1}{2}Q_4(u_{\varepsilon})\mu_{4} \\
    &<1-\frac{1}{2}K_{4,\cG} \mu_{4}+\frac{\varepsilon}{2}\mu_{4}=\frac{\varepsilon}{2}\mu_{4},
\end{split}
\end{equation*}
where we used again \eqref{critical mass p}. Since $\varepsilon>0$ was arbitrarily chosen, we infer that 
\begin{equation*}
     \mathcal{F}_{p,4,\alpha}\left(\frac{\mu_{4}}{\alpha}\right) \leq 0, \quad \text{ whence } \quad \mu_{p,4,\alpha}\leq \frac{\mu_{4}}{\alpha},
\end{equation*}
which completes the proof of the upper estimates.

The lower bound for $\mu_{p,q,\alpha}$ can be found by observing that, for $u \in H^1_{\mu}(\cG)$,
\begin{equation*}
\begin{split}
    F_{p,q,\alpha}(u, \mu)&=1-\frac{2}{p}Q_p(u)\mu^{\frac{p-2}{2}}-\frac{2\alpha}{q}Q_q(u)\mu^{\frac{q-2}{2}}\geq 1-\frac{2}{p}K_{p,\mathcal{G}}\mu^{\frac{p-2}{2}}-\frac{2\alpha}{q}K_{q,\mathcal{G}}\mu^{\frac{q-2}{2}} \\ &\geq 1-\left(\frac{\mu}{\mu_{p,\mathcal{G}}}\right)^{\frac{p-2}{2}}-\alpha\left(\frac{\mu}{\mu_{q,\mathcal{G}}}\right)^{\frac{q-2}{2}}:=f(\mu).
\end{split}
\end{equation*}
The function $f(\mu)$ is strictly decreasing for $\mu \in [0, +\infty)$, $f(0) = 1$, and $f(\mu) \to -\infty$ as $\mu \to + \infty$. Therefore, there exists a unique value $\overline{\mu}_{p, q,\alpha}$ such that $f(\overline{\mu}_{p, q,\alpha})=0$, and for any $u \in H^1_{\mu}(\cG)$ 
\begin{equation*}
    F_{p,q,\alpha}(u,\overline{\mu}_{p,q,\alpha}) \geq f(\overline{\mu}_{p,q,\alpha}) = 0.
\end{equation*}
This, by \eqref{massa critica}, finally implies that $\mu_{p,q,\alpha}\geq \overline{\mu}_{p, q,\alpha}$. 

The rest of the thesis follows straightforwardly.
\end{proof}

\section{The defocusing case}\label{sec: def}

We now analyze the defocusing NLS energy \eqref{def E} with $\alpha < 0$. Firstly, Proposition \ref{(un)bounded_defocusing} states the conditions for the boundedness of $E_{p,q,\alpha}$ in $H^1_{\mu}$.
\begin{proof}[Proof of Proposition \ref{(un)bounded_defocusing}]
The lower boundedness and the coercivity in points 1 and 3 can be proved as in Proposition \ref{(un)bounded}. To show that $\cE_{p,q,\alpha}(\mu) \le 0$ for all $\mu>0$, we consider the sequences $\{u_n\}$ constructed in the proofs of Theorem \ref{sottocritico_combined} or \ref{sottocritico bidimensionale combined}, and observe that $E_{p,q,\alpha}(u_n) \to 0$ as $n \rightarrow +\infty$ (but, since $\alpha<0$, we cannot conclude that the ground state level is negative).

Regarding the unboundedness from below when $\mu > \tilde \mu_{\cG}$, if $\cG$ has a terminal edge we can argue exactly as in \cite[Proof of Theorem 1.7 ($i$)]{combined}; moreover, the same techniques also shows that $\cE_{6,q,\alpha}(\mu, \cG)=-\infty$ if $\cG$ has no terminal edges and $\mu>\tilde \mu_{\cG} = \mu_{\R}$ (it is sufficient to start from a compactly supported function $u \in H^1_\mu(\R)$ with negative energy $E_6(u)<0$).

Finally, let $p=6$ and $\mu \in (0,\tilde\mu_{\cG}]$. By \eqref{GN}, we have that
\[
E_{6,q,\alpha}(u) \ge \frac12\left(1-\left(\frac{\mu}{\mu_{6,\cG}}\right)^2\right) \|u'\|_2^2 - \alpha \frac{C_{q,\cG}}{q} \mu^{\frac{q+2}{4}} \|u'\|_{L^2(\cG)}^{\frac{q-2}{2}}> \frac12\left(1-\left(\frac{\mu}{\mu_{6,\cG}}\right)^2\right) \|u'\|_2^2
\]
for all $u \in H^1_{\mu}(\cG)$. This implies that $\cE_{6,q,\alpha}(\mu) = 0$, and that the infimum is not attained.
\end{proof}

Observe that, unlike the focusing setting, $E_{p,q,\alpha}$ is bounded from below on $H^1_{\mu}$ when $p=6$ and $\mu=\tilde \mu_{\mathcal{G}}$. 
Proposition \ref{(un)bounded_defocusing} implies that, when $p=6$ and $\mu_{6,\mathcal{G}}=\tilde{\mu}_{\mathcal{G}}$, the ground state does not exist for every $\mu>0$. Recall that sufficient conditions for having $\mu_{6,\mathcal{G}}=\tilde{\mu}_{\mathcal{G}}$ are either that $\mathcal{G}$ contains a terminal edge, or that it satisfies condition ($H$). 

Now we focus on the existence results for either $p<6$ and $\mu>0$, or $p=6$ and $\mu_{6,\mathcal{G}}< \mu \leq \tilde{\mu}_{\mathcal{G}}$. For this inequality to hold, it is necessary that $\tilde{\mu}_{\mathcal{G}} = \mu_{\R}$. First of all, we prove a sufficient condition for the coercivity of $E_{6,q,\a}$ in $H^1_{\mu}$.
\begin{proposition}
\label{coercivity defocusing critical}
    Let $\cG$ be a $1$ or $2$-periodic metric graph, $2 < q <p=6$, and $\mu_{6,\mathcal{G}}<\mu\leq \tilde{\mu}_{\mathcal{G}}=\mu_{6,\mathbb{R}}=\mu_{\mathbb{R}}$. If $\cE_{6,q,\a}(\mu)<0$, then $E_{6,q,\a}$ is coercive in $H^1_{\mu}(\cG)$.
\end{proposition}
\begin{proof}
The condition $\mu_{6,\mathcal{G}}<\tilde{\mu}_{\mathcal{G}}=\mu_{6,\mathbb{R}}$ implies that there are no terminal edges in $\cG$. Then, we can apply the modified Gagliardo-Nirenberg inequality in Lemma \ref{lem: mod GN} on a minimizing sequence $\{u_n\}$:
\begin{equation*}
\begin{split}
    E_{6,q,\a}(u_n)&=\frac{1}{2}\|u_n'\|_2^2-\frac{1}{p}\|u_n\|_6^6-\frac{\a}{q}\|u_n\|_q^q \\
    &\geq \frac{1}{2}\|u_n'\|_2^2\left(1-\left(\frac{\mu-\theta_n}{\mu_{\R}}\right)^2\right)-C\theta_n^{\frac{1}{2}}-\frac{\a}{q}\|u_n\|_q^q\\
    &\geq \frac{1}{2}\|u_n'\|_2^2\left(1-\left(\frac{\mu_{\R}-\theta_n}{\mu_{\R}}\right)^2\right)-C\theta_n^{\frac{1}{2}}-\frac{\a}{q}\|u_n\|_q^q
\end{split}
\end{equation*}
    for some $\theta_n \in [0,\mu]$. Since $E_{6,q,\a}(u_n) \rightarrow \mathcal{E}_{6,q,\a}(\mu)<0$, there exists $C>0$ such that $\theta_n \geq C$, for $n$ sufficiently large. Therefore,
\begin{equation*}
    E_{6,q,\a}(u_n)  \geq\frac{1}{2}\|u_n'\|_2^2\left(1-\left(\frac{\mu_{\R}-C}{\mu_{\R}}\right)^2\right)-C\mu^{\frac{1}{2}} - \frac{\a}{q}C\|u_n\|_2^{\frac{q}{2}+1}\|u_n'\|_2^{\frac{q}{2}-1},
\end{equation*}    
having used the Gagliardo-Nirenberg inequality \eqref{GN} on the last term. Since $\mathcal{E}_{6,q,\a}(\mu)<0$, this proves the boundedness of $\|u_n'\|_2^2$.
\end{proof}

Next, we prove the counterparts of Lemma \ref{dich_combined} and Proposition \ref{negative_combined}.

\begin{lem}
\label{dichotomy combined defocusing}
Let $\mathcal{G}$ be a $1$ or $2$-periodic metric graph, $2 < q <p$, with either $p<6$ and $\mu>0$, or $p=6$ and $\mu_{6,\mathcal{G}}<\mu\leq \tilde{\mu}_{\mathcal{G}}=\mu_{6,\mathbb{R}}$. Suppose further that $\cE_{p,q,\a}(\mu)<0$, and let $\{u_n\} \subseteq H^1_{\mu}$ be a minimizing sequence for $E_{p,q,\alpha}$. Then up to a subsequence $u_n \rightharpoonup u$ in $H^1$, and either $u \in H^1_{\mu}$ is a ground state of $E_{p,q,\a}$, or $u\equiv 0$ and $\|u\|_{L^\infty}(\cG) \to 0$.
\end{lem}

The proof of the lemma borrows several ingredients from analogue statements proved in \cite{threshold, square_grid, combined}. However (unlike Lemma \ref{dich_combined}) it is not a direct adaptation, and it needs further justifications. 
Thus, we provide a sketch of the proof, for the sake of completeness.

\begin{proof}
As shown in \cite[Lemma 3.5]{combined} for $\alpha>0$, and in \cite[Theorem 3.2]{threshold} or \cite[Lemma 3.2]{square_grid} for $\alpha=0$, the lemma can be proved by employing the Brezis-Lieb lemma on the energy $E_{p,q,\alpha}$, combined with the subadditivity inequality on $\cE_{p,q,\alpha}$ in $H^1_\mu$: for every $\mu_1,\mu_2>0$ with $\mu_1+\mu_2=\mu$, it results that
\beq
\label{subadditivity}
\cE_{p,q,\a}(\mu) < \cE_{p,q,\a}(\mu_1)+\cE_{p,q,\a}(\mu_2)
\eeq
With these ingredients in hand, the proof of Lemma \ref{dichotomy combined defocusing} is complete. Regarding the Brezis-Lieb lemma, it holds on periodic graphs, without any problem. Thus, it remains to show that under the current assumptions the subadditivity inequality \eqref{subadditivity} is valid on $\mathcal{E}_{p,q,\a}$. In this perspective, for us it is essential to suppose that $\cE_{p,q,\alpha}(\mu)<0$. 

First of all, we claim that for all $\nu \in (0,\mu)$ and $\theta \ge 1$ such that $\theta \nu \le \mu$
\begin{equation}
\label{disuguaglianza theta}
\mathcal{E}_{p,q,\alpha}(\theta \nu) \leq \theta^\frac{p}{2} \mathcal{E}_{p,q,\alpha}(\nu).
\end{equation}
Indeed, let $\{u_n\} \subseteq H^1_{\nu}$ be a minimizing sequence for $E_{p,q,\alpha}$. Then, $\{\theta^{\frac{1}{2}}u_n\} \subseteq H^1_{\theta \nu}$, and
\begin{equation*}
\begin{split}
     \mathcal{E}_{p,q,\alpha}(\theta \nu) &\leq E_{p,q,\alpha}(\theta^{\frac{1}{2}}u_n) = \frac{1}{2}\theta^2\|u_n'\|_2^2-\frac{1}{p}\theta^{\frac{p}{2}}\|u_n\|_p^p+\frac{|\a|}{q}\theta^{\frac{q}{2}}\|u_n\|_q^q \\
     &\leq \theta^{\frac{p}{2}}E_{p,q,\a}(u_n) = \theta^{\frac{p}{2}}\left(\mathcal{E}_{p,q,\a}(\nu)+o(1)\right)
\end{split}
\end{equation*}
whence \eqref{disuguaglianza theta} follows. By choosing $\theta=\mu/\mu_1 = (\mu_1+\mu_2)/\mu_1$ and $\nu=\mu_1$ in \eqref{disuguaglianza theta}, we obtain
\[
\mathcal{E}_{p,q,\a}(\mu) \leq \left(\frac{\mu_1+\mu_2}{\mu_1}\right)^{\frac{p}{2}}\mathcal{E}_{p,q,\a}(\mu_1) \quad \iff \quad \left(\frac{\mu_1}{\mu_1+\mu_2}\right)^{\frac{p}{2}} \mathcal{E}_{p,q,\a}(\mu) \le \mathcal{E}_{p,q,\a}(\mu_1).
\]
Replacing $\mu_1$ with $\mu_2$, we also have
\[
\left(\frac{\mu_2}{\mu_1+\mu_2}\right)^{\frac{p}{2}}\mathcal{E}_{p,q,\a}(\mu_1+\mu_2) \leq \mathcal{E}_{p,q,\a}(\mu_2).
\]
Summing side by side the previous inequalities, we obtain
\begin{equation*}
    \left(\left(\frac{\mu_1}{\mu_1+\mu_2}\right)^{\frac{p}{2}}+\left(\frac{\mu_2}{\mu_1+\mu_2}\right)^{\frac{p}{2}}\right)\mathcal{E}_{p,q,\a}(\mu) \leq \mathcal{E}_{p,q,\a}(\mu_1)+\mathcal{E}_{p,q,\a}(\mu_2).
\end{equation*}
As a consequence, using also that $\mathcal{E}_{p,q,\a}(\mu)<0$ and $\mu_1, \mu_2>0$,
\begin{equation*}
\begin{split}
    \mathcal{E}_{p,q,\a}(\mu) &= \frac{\mu_1}{\mu_1+\mu_2}\mathcal{E}_{p,q,\a}(\mu)+\frac{\mu_2}{\mu_1+\mu_2}\mathcal{E}_{p,q,\a}(\mu) \\
    &< \left(\frac{\mu_1}{\mu_1+\mu_2}\right)^{\frac{p}{2}}\mathcal{E}_{p,q,\a}(\mu)+\left(\frac{\mu_2}{\mu_1+\mu_2}\right)^{\frac{p}{2}}\mathcal{E}_{p,q,\a}(\mu) \\
    &\leq\mathcal{E}_{p,q,\a}(\mu_1)+\mathcal{E}_{p,q,\a}(\mu_2),
    \end{split}
\end{equation*}
which is the desired inequality. As already discussed, at this point one can proceed as in \cite{threshold, combined}. 
\end{proof}

As in \cite{threshold, square_grid, combined}, one can deduce the:

\begin{proposition}
\label{negative_defocusing}
Let $\mathcal{G}$ be a $1$ or $2$-periodic metric graph, $2 < q <p$, with either $p<6$ and $\mu>0$, or $p=6$ and $\mu_{6,\mathcal{G}}<\mu\leq \tilde{\mu}_{\mathcal{G}}=\mu_{6,\mathbb{R}}$. If $\mathcal{E}_{p,q,\alpha}(\mu)<0$, then there exists a ground state of mass $\mu$.
\end{proposition}

In what follows, we focus on Theorem \ref{defocusing_theorem}. It is clear that the function $\alpha \mapsto \mathcal{E}_{p,q,\alpha}(\mu)$ is non-increasing for $\alpha \in (-\infty, 0]$ (since the same property holds for  $\alpha \mapsto E_{p,q,\alpha}(u)$, for every $u \in H^1_\mu(\cG)$ fixed). Thus, the special value $\overline{\alpha}$ introduced in \eqref{alpha_definition} is well defined (as already observed, the definition can also be given on noncompact metric graphs with a finite number of edges, see \cite[Lemma 4.1]{combined}), and we give a useful alternative characterization.
\begin{lem}
	\label{caratterizzazione alpha}
	Let $\mathcal{G}$ be either a periodic graph, or a noncompact graph with a finite number of edges; let $\mu>0$ and $2<q<p \leq 6$. Additionally, suppose that $\mathcal{E}_p(\mu)<0$. Then:
	\begin{equation}
		\label{alpha_definition_2}
		\overline{\alpha}= \inf\limits_{u \in H^1_{\mu}}Q_{p,q}(u), \quad \text{being} \quad
		Q_{p,q}(u):=\frac{\frac{1}{2}\|u'\|_2^2-\frac{1}{p}\|u\|_p^p}{\frac{1}{q}\|u\|_q^q}.
	\end{equation}
\end{lem}
\begin{proof}
	We denote by $\tilde \alpha$ the right hand side in \eqref{alpha_definition_2}. Since $\mathcal{E}_p(\mu)<0$, we have that $\tilde{\alpha}<0$. First of all, we prove that $\tilde{\alpha}\leq \overline{\alpha}$. If $\tilde{\alpha}=-\infty$, there is nothing to prove; if $\tilde{\alpha}>-\infty$, then
	\begin{equation*}
	\tilde{\alpha} \leq Q_{p,q}(u) \quad \implies \quad  	E_{p,q,\tilde{\alpha}}(u) = \frac{1}{2}\|u'\|_2^2-\frac{1}{p}\|u\|_p^p-\frac{\tilde{\alpha}}{q}\|u\|_q^q \geq 0  \qquad \forall u \in H^1_{\mu},
	\end{equation*}
	whence $\mathcal{E}_{p,q,\tilde{\alpha}}(\mu) = 0$. Thus, by definition \eqref{alpha_definition}, we conclude that $\tilde{\alpha}\leq \overline{\alpha}$. \\
	Now, we address the opposite inequality. Let $\beta>\tilde{\alpha}$ (if $\tilde{\alpha}=-\infty$, we take any $\beta<0$). By definition of $\tilde \alpha$, we can assert that
	\begin{equation*}
		\exists u_{\beta} \in H^1_{\mu} \text{ such that} \qquad \beta > \frac{\frac{1}{2}\|u_{\beta}'\|_2^2-\frac{1}{p}\|u_{\beta}\|_p^p}{\frac{1}{q}\|u_{\beta}\|_q^q},
	\end{equation*}
	namely $E_{p,q,\beta}(u_{\beta})<0$. This shows that $\mathcal{E}_{p,q,\beta}(\mu)<0$, and therefore $\beta\geq \overline{\alpha}$. Since $\beta$ is a generic value above $\tilde{\alpha}$, we conclude that $\overline{\alpha}\leq \tilde{\alpha}$.
\end{proof}

\begin{proof}[Proof of Theorem \ref{defocusing_theorem}] 
(1) Clearly, $\overline{\alpha}=0$ if $\mathcal{E}_p(\mu)=0$. On the other hand, if $\mathcal{E}_p(\mu)<0$ we have $\overline{\alpha}<0$. Indeed, $\mathcal{E}_{p}(\mu)<0$ implies that there exists $u \in H^1_{\mu}$ such that $E_{p}(u)<0$, and by continuity with respect to $\alpha$ the same holds for $E_{p,q,\alpha}(u)$ provided that $|\alpha|$ is sufficiently small. \\
(2) The existence of a ground state for $\alpha \in (\overline{\alpha},0)$ follows directly by Definition \ref{alpha_definition}, and by Proposition \ref{negative_defocusing}. The proof of the non-existence for $\alpha<\overline{\alpha}$ is presented in \cite[Theorem 1.7 (ii) (b)]{combined} for $p=6$, and holds without modifications on periodic graphs, also when $p<6$.\\
(3) 
This point concerns the case when $\overline{\alpha}<0$; thus, by point 1, $\cE_p(\mu) <0$, so that the alternative characterization in Lemma \ref{caratterizzazione alpha} is available. First of all, by continuity of $E_{p,q,\alpha}(u)$ (for fixed $u$) with respect to $\alpha$, it follows easily that $\mathcal{E}_{p,q,\overline{\alpha}}(\mu)=0$; otherwise, we would obtain a contradiction with the definition \eqref{alpha_definition} of $\overline{\alpha}$ as $\inf$. 
Next, 
we prove that $\overline{\alpha}>-\infty$ and that any minimizing sequence $\{u_n\}$ for $Q_{p,q}$ is bounded in $H^1$. If $p<6$, then 
\begin{equation*}
	\begin{split}
		Q_{p,q}(u_n)&=\frac{\frac{1}{2}\|u_n'\|_2^2-\frac{1}{p}\|u_n\|_p^p}{\frac{1}{q}\|u_n\|_q^q} \geq C\frac{\|u_n'\|_2^2}{\|u_n'\|_2^{\frac{q}{2}-1}\|u_n\|_2^{\frac{q}{2}+1}}-\frac{q}{p}\frac{\|u_n\|_q^q\|u_n\|_{\infty}^{p-q}}{\|u_n\|_q^q} \\
		&  \geq C\|u_n'\|_2^{3-\frac{q}{2}}-C\|u_n'\|_2^{\frac{p}{2}-\frac{q}{2}}\|u_n\|_2^{\frac{p}{2}-\frac{q}{2}}		
	\end{split}
\end{equation*}
for some $C>0$, where we used the Gagliardo-Nirenberg inequalities \eqref{GN} and \eqref{GN_infty}. Since $3-\frac{q}{2}>\frac{p}{2}-\frac{q}{2}$, this proves that $\overline{\alpha}>-\infty$ and that $\|u_n'\|_2^2$ is bounded. \\
Regarding the case $p=6$, we can focus on $\mu\in(\mu_{\cG}, \mu_{\mathbb{R}}]$. First of all, we observe that 
\begin{equation}
\label{disug su u_n non fa vanishing}
Q_{6,q}(u_n)=\frac{\frac{1}{2}\|u_n'\|_2^2}{\frac{1}{q}\|u_n\|_q^q}-\frac{\frac{1}{6}\|u_n\|_6^6}{\frac{1}{q}\|u_n\|_q^q} \geq -\frac{q}{6}\frac{\|u_n\|_6^6}{\|u_n\|_q^q} \geq -\frac{q}{6}\frac{\|u_n\|_q^q\|u_n\|_{\infty}^{6-q}}{\|u_n\|_q^q}=-\frac{q}{6}\|u_n\|_{\infty}^{6-q}.
\end{equation}	
Since $Q_{6,q}(u_n)\rightarrow \overline{\alpha}<0$, 
\begin{equation}
	\label{u_n non fa vanishing}
\exists C>0 \text{ such that } \|u_n\|_{\infty}>C \text{ for sufficiently large $n$}.
\end{equation}
Moreover, since $\overline{\a}$ is negative, the numerator of $Q_{6,q}(u_n)$ is (eventually) negative as well, and the Gagliardo-Nirenberg inequality \eqref{GN_infty} implies that 
\[
\exists C>0 \text{ such that } \frac{1}{6}\|u_n\|_6^6>\frac{1}{2}\|u_n'\|_2^2 \geq C \|u_n\|_\infty^4 \text{ for sufficiently large $n$.}
\]
This estimate leads to
\begin{equation}
\label{u_n_q non fa vanishing}
\|u_n\|_q^q \geq \frac{\|u_n\|_6^6}{\|u_n\|_{\infty}^{6-q}} \geq C\frac{\|u_n\|_{\infty}^4}{\|u_n\|_{\infty}^{6-q}}=C\|u_n\|_{\infty}^{q-2} \geq C,
\end{equation}
eventually as $n \rightarrow +\infty$, where we used \eqref{u_n non fa vanishing} and the fact that $q-2>0$. By Lemma \ref{caratterizzazione alpha}, we deduce that 
\[
\overline{\alpha}=\lim\limits_{n}\frac{E_6(u_n)}{\frac{1}{q}\|u_n\|_q^q} \geq \frac{\mathcal{E}_6(\mu)}{C}>-\infty,
\]
for some $C>0$, by Theorems \ref{1_periodic_critical}-($iii$) and \ref{teorema critico grafi 2-periodici}-($iii$). It remains to show that a minimizing sequence for $Q_{p,q}$ in $H^1_\mu$ is bounded. By the modified Gagliardo-Nirenberg inequality in Lemma \ref{lem: mod GN} and \eqref{u_n_q non fa vanishing}
\[
\overline{\alpha} \leftarrow Q_{6,q}(u_n)=\frac{\frac{1}{2}\|u_n'\|_2^2-\frac{1}{6}\|u_n'\|_6^6}{\frac{1}{q}\|u_n\|_q^q} \geq \frac{\frac{1}{2}\|u_n'\|_2^2\left(1-\left(\frac{\mu-\theta_n}{\mu_{\mathbb{R}}}\right)^2\right)-C\theta_n^{\frac{1}{2}}}{C} \geq -C\theta_n^{\frac{1}{2}},
\]
 for some $C>0$ and $\theta_n \in [0,\mu]$. Since $\overline{\alpha}<0$, we infer that $\theta_n \ge C$. Therefore, $Q_{6,q}(u_n)$ can be estimated as follows:
 \[
Q_{6,q}(u_n)\geq \frac{\frac{1}{2}\|u_n'\|_2^2\left(1-\left(\frac{\mu-\theta_n}{\mu_{\mathbb{R}}}\right)^2\right)-C\theta_n^{\frac{1}{2}}}{C} \geq C\|u_n'\|_2^2\left(1-\left(\frac{\mu-C}{\mu_{\mathbb{R}}}\right)^2\right)-C\mu^\frac{1}{2},
 \]
where the coefficient of $\|u_n'\|_2^2$ is strictly positive. Hence, $\|u_n'\|_2^2$ must be bounded, as desired.

At this point, both for $p<6$ and for $p=6$, we consider a bounded minimizing sequence $\{u_n\} \subseteq H^1_{\mu}$ for $Q_{p,q}$. The fact that $Q_{p,q}(u_n) \to \overline{\alpha}$ 
implies that $E_{p,q,\overline{\alpha}}(u_n) \to 0=\mathcal{E}_{p,q,\overline{\alpha}}(\mu)$. Hence, $\{u_n\} \subseteq H^1_{\mu}$ is a bounded minimizing sequence for $E_{p,q,\overline{\alpha}}$. Up to a subsequence, $u_n \rightharpoonup u$ weakly in $H^1$. By Lemma \ref{dichotomy combined defocusing}, we deduce that either $u$ is a ground state of mass $\mu$ for $E_{p,q,\overline{\alpha}}$, or $u \equiv 0$ and $\|u\|_{L^\infty(\cG)} \to 0$. We prove that the case $u \equiv 0$ cannot occur. If $u_n \rightharpoonup 0$, then 
\begin{equation*}
	\overline{\alpha} = \lim\limits_{n}\frac{\frac{1}{2}\|u_n'\|_2^2-\frac{1}{p}\|u_n\|_p^p}{\frac{1}{q}\|u_n\|_q^q}\geq \liminf\limits_n\frac{q}{2}\frac{\|u_n'\|_2^2}{\|u_n\|_q^q}-\frac{q}{p}\frac{\|u_n\|_q^q}{\|u_n\|_q^q}\|u_n\|_{\infty}^{p-q}= \liminf\limits_n\frac{q}{2}\frac{\|u'_n\|_2^2}{\|u_n\|_q^q}\geq 0
\end{equation*}
which is a contradiction, since $\overline{\alpha}<0$. Therefore, $u$ is a ground state of mass $\mu$ for $E_{p,q,\overline{\alpha}}$. 
\end{proof}

The proof of the point 3 can be adapted also to metric graphs $\mathcal{G}$ with a finite number of edges.

\begin{proof}[Proof of Theorem \ref{alpha critical 1}]
Let $K$ be the compact core of $\mathcal{G}$, given by the union of the edges of finite measure, and let $\{e_i\}$ be the half-lines of $\mathcal{G}$, with $i \in \{1,2, \dots, k\}$, $k \in \mathbb{N}$. Name $0_i$ the starting point of each half-line $e_i$ and, for $P \in e_i$, assign coordinates $x_P \in [0, +\infty)$ such that $x_P=0$ if and only if $P=0_i$. \\
Proving that $Q_{6,q}$ admits a ground state in $H^1_{\mu}$ implies that $\overline{\alpha}>-\infty$, and that there exists a minimizer for $\cE_{6,q,\overline{\alpha}}(\mu)=0$. Let $\{u_n\} \subseteq H^1_{\mu}$ be the minimizing sequence for $Q_{6,q}$. As in the previous proof, $\{u_n\}$ is bounded in $H^1$ and $\|u_n\|_q^q \ge C>0$; suppose by contradiction that $u_n \rightharpoonup 0$, from which $u_n \rightarrow 0$ in $L^{\infty}(K)$. Therefore, $u_n \rightarrow 0$ in $L^p(K)$ for each $p \geq 2$, and $u_n(0_i) \leq 1$ for each $i$, for $n$ sufficiently large. 
At this point, we define $k$ functions in $H^1(\mathbb{R})$, obtained from the restrictions ${u_n}_{|{e_i}}$. 
\begin{equation*}
\text{For each $i \in \{1 \dots k\}$,} \quad    u^i_n(x):=\begin{cases}
        0 \quad &\text{ if } x<-1,
        \\u_n(0_i)(x+1) \quad &\text{ if } -1\leq x <0, \\
        {u_n}_{|e_i}(x) & \text{ if } x \geq 0,
    \end{cases}
\end{equation*}
so that:
\begin{itemize}
    \item for each $i$, $u^i_n \in H^1(\mathbb{R})$, since the continuity condition holds;
    \item $\sum_{i}\|u^i_n\|_p^p=\|u_n\|_p^p+o(1)$ for all $p \geq 2$, since $\|u_n\|_{L^p(K)}^p\rightarrow 0$ and $\|u^i_n\|_{L^p([-1,0])}^p \rightarrow 0$  
    \item $\sum_{i}\|(u^i_n)'\|_2^2\leq \|u'_n\|_2^2+o(1)$, since
    \begin{equation*}
        \sum_i\|(u^i_n)'\|_2^2 = \sum_i\|(u^i_n)'\|_{L^2([0, +\infty))}^2+o(1) = \sum_i\|u_n'\|_{L^2(e_i)}^2+o(1) \leq \|u_n'\|_2^2 + o(1)
    \end{equation*}
    (where we used that $\|(u^i_n)'\|_{[-1,0]}^2 \rightarrow 0$ in the first equality).
    \item $\|u^i_n\|_2^2 \le \mu_\R+o(1)$ for every $i$, since 
    \begin{equation}\label{1511}
    \|u^i_n\|_2^2 \leq \sum\limits_{i}\|u^i_n\|_2^2=\|u_n\|_2^2+o(1) \leq \mu_{\mathbb{R}}+o(1).
\end{equation}
\end{itemize}
With this transformed sequence in hand, we conclude that 
\begin{equation}
\label{8.24}
\begin{split}
    Q_{6,q}(u_n) &\geq \frac{\frac{1}{2}\sum_{i}\|(u^i_n)'\|_2^2-\frac{1}{6}\sum_{i}\|u^i_n\|_6^6 +o(1)}{\frac{1}{q}\|u_n\|_q^q} \\
    &= \frac{\sum_i \frac{1}{2}\|(u_n^i)'\|_2^2\left(1-\frac{\|u_n^i\|_6^6}{\|(u_n^i)'\|_2^2\|u_n^i\|_2^4}\|u_n^i\|_2^4\right) +o(1)}{\frac{1}{q}\|u_n\|_q^q}\\
    &\geq \frac{\sum_i \frac{1}{2}\|(u_n^i)'\|_2^2\left(1-\left(\frac{\|u_n^i\|_2^2}{\mu_{\mathbb{R}}}\right)^2\right) +o(1)}{\frac{1}{q}\|u_n\|_q^q} \geq \frac{o(1)}{\|u_n\|_q^q} =o(1),
\end{split}
\end{equation}
in contradiction with the fact that $Q_{6,q}(u_n)\rightarrow \overline{\alpha}<0$. 
In the second to last step, we use definition \eqref{critical mass} with $\cG=\R$, and, in the last step, we used \eqref{1511} and the fact that $\|u_n\|_q \ge C$. 
\end{proof}

\medskip

\noindent \textbf{Data availability:} No data were used for the research described in the article.

\end{document}